\documentclass[amsfonts,onesided,11pt]{amsart}

\usepackage{amssymb,amsmath,amsfonts, amscd}
\usepackage[all,arc,cmtip]{xy}
\usepackage[bookmarks]{hyperref}
\usepackage{tikz}
\usepackage{enumerate}
\usepackage{mathrsfs}
\usepackage[margin=1in]{geometry}

\allowdisplaybreaks 

\newtheorem{thm}{Theorem}[section]
\newtheorem{cor}[thm]{Corollary}
\newtheorem{prop}[thm]{Proposition}
\newtheorem{lem}[thm]{Lemma}
\newtheorem{conj}[thm]{Conjecture}

\theoremstyle{definition}

\theoremstyle{remark}
\newtheorem{rem}[thm]{Remark}

\makeatletter
\let\c@equation\c@thm
\makeatother
\numberwithin{thm}{subsection}
\numberwithin{equation}{subsection}
\begin{document}
\everymath{\displaystyle} 
\title{Support varieties of line bundle cohomology groups for $SL_3(k)$}
\author{William D. Hardesty}
\address
{Department of Mathematics\\ University of Georgia \\
Athens\\ GA~30602, USA}
\email{hardesty@math.uga.edu}
\date{\today}

\begin{abstract}
  Let $G= SL_3(k)$ where $k$ is a field of characteristic $p > 0$
  and let $\lambda \in X(T)$ be any weight with corresponding line bundle $\mathscr{L}(\lambda)$ on $G/B$. 
  In this paper we compute the support varieties for all modules of the form 
  $H^i(\lambda):= H^i(G/B, \mathscr{L}(\lambda))$ over the first Frobenius kernel 
  $G_1$. The calculation involves certain recursive character formulas given
   by Donkin (cf. \cite{DONK2007}) which can be used to compute the characters of the line bundle cohomology groups. In the case where 
  $\lambda$ is a $p$-regular weight and $M=H^i(\lambda)\neq 0$ for some $i$,
  these formulas are used to show that any $p^{th}$ root of unity $\zeta$ is not
  a root of the generic dimension of $M$. To handle the case where $\lambda$ is not  $p$-regular, we employ techniques 
  similar to those used by Drupieski, Nakano and Parshall (cf. \cite{DNP}) to show that the
  module $H^i(\lambda)$ is not projective over $G_1$ whenever it is  nonzero and $\lambda$ lies outside of the Steinberg block.
\end{abstract}
\maketitle

\section{Introduction}
\subsection{ }
Let $G$ be a semisimple simply connected algebraic group over an algebraically closed field $k$ of charactersitic $p >0$.
 It is well known
that any finite dimensional $B$-module $M$ induces a vector bundle $\mathscr{L}(M)$ of rank $m:=\dim_kM$ on the flag manifold $G/B$.
An important problem in representation theory is to try to understand the structure of
the sheaf cohomology groups $H^i(M) :=H^i(G/B,\mathscr{L}(M))$ for any finite dimensional $B$-module $M$.

The most well studied case is when $M=\lambda$ for $\lambda \in X(T)$.
In particular, when $\lambda \in X(T)_+$, we know by Kempf's vanishing theorem that $H^i(\lambda) = 0$ for $i > 0$ and
that the character of $H^0(\lambda)$ is given by Weyl's character formula (cf. \cite[II.5]{jan2003}). 
For $\lambda \not\in X(T)_+$, the vanishing behavior of
$H^1(\lambda)$ has been completely determined by Andersen 
in \cite[Theorem 3.6]{and21979}. 
For arbitrary $\lambda \in X(T)$, little is known 
about the structure of the modules $H^i(\lambda)$. 
However, in \cite[Theorem 6.2.1]{npv2002}, the 
\emph{support varieties} $V_{G_1}(H^0(\lambda))$ were determined for all $\lambda \in X(T)_+$ when the underlying
field has good characteristic. 

In this paper we will determine $V_{G_1}(H^i(\lambda))$ for any $i$ and $\lambda \in X(T)$ when
$G=SL_3(k)$.  The technique will be to  bound the order in which 
$\psi_p(t)=1 + t + \cdots + t^{p-1}$ divides the generic dimension $\dim_tH^i(\lambda)\in\mathbb{Z}[t,t^{-1}]$
by using of the recursive character formulas given in \cite[Sections 4-6]{DONK2007} for 
the modules $H^i(\lambda)$. 

The following theorem is the main result of this paper
\begin{thm}\label{thm:main_result}
  Let $p$ be any prime and let $G=SL_3(k)$ where $\text{char}(k)=p$. If 
  $\lambda \in X(T)$ is any weight with $w\cdot \lambda \in X(T)_+$ for some $w \in W$ then
  \[
    V_{G_1}(H^i(\lambda)) = V_{G_1}(H^0(w\cdot \lambda))
  \]
  provided that $H^i(\lambda) \neq 0$.
\end{thm}

The proof of this theorem will be broken up into several sections. 
In Section~\ref{sec:gen_dim}, we will define the generic dimension of any rational $T$-module $V$ and
present some well known properties. The support variety for any module over a finite $k$-group will
be introduced  in Section~\ref{sec:supp_var}. Some general results including Proposition~\ref{prop:npv_result}, establishing the 
left inclusion of the main theorem stated above, will also be presented.

The rest of the paper will deal with the right inclusion, which will require a great deal of computation. 
The recursive character formulas mentioned above will be given in a correct, 
simplified form in Section~\ref{sec:dim_formulas}. We will employ these formulas to determine the support varieties for
all $p$-regular weights in Section~\ref{sec:reg_case}, by transforming them into 
simpler recursive quantum dimension formulas which
are then shown to be non-vanishing (cf. Theorem~\ref{thm:reg_result}). The case where $p\geq 3$ and
$\lambda$ is not $p$-regular is dealt with in Section~\ref{sec:sreg_case} by proving that 
the modules $H^i(\lambda)$ are not projective over $G_1$ whenever  $\lambda$ lies outside of the Steinberg block 
(cf. Proposition~\ref{thm:proj_result}). For clarity, a number of the technical lemmas which are used to 
establish the subregular result have been shifted back to Section~\ref{sec:verif}.

In Section~\ref{sec:alt}, we show that a simpler alternative technique can be used to compute 
the support varieties for many types of weights. In particular, this technique will handle the  $p=2$ case
of Theorem~\ref{thm:main_result}. Thus the theorem will be verified for all primes $p$. 
Finally, in Section~\ref{sec:conj} we will present an open conjecture regarding the case when $G$ is any 
semisimple simply connected algebraic group. 

\subsection{Notation}
The root system of $G$ with respect to a maximal torus $T$ will be denoted by $\Phi$, and its basis will be denoted by
$\Pi = \{\alpha_1,\dots,\alpha_n\}$. We shall denote by $\Phi^+$, the set of positive roots with respect to $\Pi$.
The Weyl group of $\Phi$ will be denoted by $W= N_G(T)/T$ and the affine Weyl group will be denoted by $W_p$. 
The affine Weyl group can be expressed as 
\[
  W_p = W \ltimes p\mathbb{Z}\Phi.
\]

We will denote by $B \supseteq T$  the Borel subgroup induced by the negative roots $-\Phi^+$.
The character group of $T$ will be denoted $X(T)=X(B)$. We define the \emph{formal character} of a $T$-module $V$ to be 
\[
  \text{ch}\, V = \sum_{\lambda \in X(T)}(\dim V_{\lambda})\,e(\lambda) \in \mathbb{Z}[X(T)]
\]
where $V_{\lambda} = \{ v \in V \, \mid \, t\cdot v = \lambda(t)v \text{ for all $t \in T$} \}$.
For each $\lambda \in X(T)$, we will identify $\lambda = k_{\lambda}$ with the one dimensional $B$-module where 
$b\cdot 1 = \lambda(b)1$ for any $b \in B$. 
The affine Weyl group acts naturally on $X(T)$ by affine transformations, where each $w \in W_p$ can be viewed as an operator $w: X(T) \rightarrow X(T)$.
This induces an action on the ring $\mathbb{Z}[X(T)]$ via $w\cdot e(\lambda) = e(w(\lambda))$ for $w \in W_p$ and $\lambda \in X(T)$. It
is well known that the invariant ring $\mathbb{Z}[X(T)]^W$ is isomorphic to the Grothendieck ring
for $G$ (cf. \cite[II.5]{jan2003}). 

The set of
dominant integral weights is given by
\[
  X(T)_+=\{\lambda \in X(T) \, \mid \, \langle\lambda,\alpha^{\vee}\rangle \geq 0 \text{ for all $\alpha \in \Pi$} \}.
\]
For each $r \geq 1$, we can define the set of $p$-restricted weights by 
\[
  X_r(T) = \{ \lambda \in X(T) \, \mid \, 0 \leq \langle\lambda,\alpha^{\vee}\rangle <p^r\text{ for all $\alpha \in \Pi$} \}.
\]

Define $\rho = \frac{1}{2}\sum_{\alpha\in\Phi}\alpha$ and the 
\emph{Coxeter number}  $h = \text{max}(\rho,\alpha_{0}^{\vee})+1$ where 
$\alpha_0 \in \Phi$ is the maximal short root if $\Phi$ is indecomposable,
otherwise we take the maximum of the Coxeter numbers over each irreducible
component of $\Phi$ (cf. \cite[II.6.2]{jan2003}). 
The affine Weyl group $W_p$ also acts on $X(T)$ by the \emph{dot action}, which is given by 
\[
  w\cdot \lambda = w(\lambda + \rho)-\rho.
\]
Two weights $\lambda, \mu \in X(T)$ are said to be $\emph{linked}$ if $\lambda = w\cdot \mu$ for some $w \in W_p$. 
We call the weight $\xi = (p-1)\rho$ the \emph{Steinberg weight}, and we shall say that a weight $\lambda \in X(T)$
is in the \emph{Steinberg block} if $\lambda = w\cdot \xi$ for some $w \in W_p$.

We also define the numbers $d_{\alpha} = \langle\alpha,\alpha\rangle/\langle \alpha_0, \alpha_0 \rangle$
for any $\alpha \in \Phi$ so that if $C = (\langle \alpha_i,\alpha_j^{\vee}\rangle)$ is the Cartan matrix
and $d_i = d_{\alpha_i}$, then the matrix $C\cdot\text{diag}(d_1,d_2,\dots,d_n)$ is symmetric. 

A prime $p$ is called \emph{good} for $G$ if $p>2$ when $G$ has a component of type $B$,$C$ or $D$;
$p>3$ when $G$ has a component of type $G_2$, $F_4$, $E_6$ or $E_7$ and $p>5$ when $G$ has component of
type $E_8$, otherwise $p$ is called \emph{bad}.  

Assuming that $p$ is good, we define for each $\lambda \in X(T)$ the subroot system
\[
  \Phi_{\lambda} = \{\alpha \in \Phi\, \mid \, d_{\alpha}\langle \lambda + \rho,\alpha^{\vee}\rangle \in p\mathbb{Z} \}.
\]
If $|\Phi_{\lambda}| = \emptyset$, then $\lambda$ is called \emph{regular} (or \emph{$p$-regular}), otherwise $\lambda$ is called \emph{subregular}.
If any two weights $\lambda$, $\mu$ are linked, then $\Phi_{\lambda} = w(\Phi_{\mu})$
for some $w \in W$. 
It can be shown that $\lambda$ is in the Steinberg block if and only if $\Phi_{\lambda} = \Phi$. 
We also see that there is a subset $I \subset \Pi$ and $w \in W$ such that $w(\Phi_{\lambda})=\Phi_I$
where $\Phi_I \subset \Phi$ is the subroot system generated by $I$. 

Let $H$ be any algebraic group defined over $k$ which has an $\mathbb{F}_p$-structure, and let $F: H \rightarrow H$ denote the \emph{geometric Frobenius endomorphism} of 
$H$ (cf. \cite[I.9.2]{jan2003}). 
The Frobenius kernel is defined as $H_1=\text{ker}(F)$. 
For any $H$-module $V$, let $V^{(1)}$ be the $H$-module whose action is twisted by $F$ so that 
$h \in H$ acts on $V^{(1)}$ as $F(h)$ acts on $V$.

We will also define the endomorphism
\begin{align*}
  (-)^F: \mathbb{Z}[X(T)] &\rightarrow \mathbb{Z}[X(T)] \\
               e(\lambda) &\mapsto e(p\lambda)
\end{align*}
For any $T$-module $V$ that 
\[
  \text{ch}\,V^{(1)}= (\text{ch}\,V)^F. 
\]

\subsection{Acknowledgements} 
This paper is a part of the author's PhD dissertation and he would like thank
his PhD thesis advisor, Daniel Nakano, for all of his consultation during this project.
The author also thanks Stephen Donkin, Henning Haahr Andersen, James Humphreys and others for their feedback and suggestions. Additionally, the author appreciates the assistance of Jun Zhang in verifying some of the identities. 

\section{The generic dimension of a module}\label{sec:gen_dim}
\subsection{ }
For any  $\lambda \in X(T)$, we can write 
 $ \lambda= \sum_{i=1}^n n_i\alpha_i$
for some  $n_i \in \mathbb{Q}$. If we now set
\[
  \text{wht}(\lambda) = \sum d_in_i \in \mathbb{Q},
\]
it can be verified that
\[
  \text{wht}(\lambda) = \frac{2\langle\lambda,\rho\rangle}{\langle\alpha_0,\alpha_0\rangle} 
  = \frac{1}{2}\sum_{\alpha \in \Phi^+}d_{\alpha}\langle\lambda,\alpha^{\vee}\rangle \in (1/2)\mathbb{Z},
\]
and hence $2\,\text{wht}(\lambda) \in \mathbb{Z}$. Thus, we can define the 
\emph{weighted height function} 
\begin{align*}
  \text{wht}: X(T) &\rightarrow (1/2)\mathbb{Z} \\
             \lambda&\mapsto \text{wht}(\lambda).
\end{align*}

For any $T$-module $V$, define the \emph{generic dimension}
\[
  \dim_tV = \sum (\dim V_{\lambda})t^{-2\,\text{wht}(\lambda)} \in \mathbb{Z}[t,t^{-1}],
\]
(cf. \cite[Section 3.1.1]{npv2002}). Observe from the definition that $\dim V = \dim_1 V$.

The generic dimensions of the induced modules are computed by \emph{Weyl's generic dimension formula}
  \[
    \dim_{t}H^0(\lambda) = \frac{\prod_{\alpha \in \Phi^+} t^{d_{\alpha}(\lambda + \rho,\alpha^{\vee})} 
    - t^{-d_{\alpha}(\lambda + \rho,\alpha^{\vee})}}{\prod_{\alpha \in \Phi^+} t^{d_{\alpha}(\rho,\alpha^{\vee})} 
    - t^{-d_{\alpha}(\rho,\alpha^{\vee})}}
  \]
(cf. \cite[Section 3.1.2]{npv2002}).

\subsection{Properties of Generic Dimension}
Define the ring homomorphism 
\begin{align*}
  \varphi: \mathbb{Z}[X(T)] &\rightarrow \mathbb{Z}[t,t^{-1}] \\
           e(\lambda) &\mapsto t^{-2\,\text{wht}(\lambda)}. 
\end{align*}
This homomorphism allows us to transfer many of the useful properties for characters
to properties for the generic dimension. 

\begin{prop}
  Let $V$, $W$ be $X(T)$-graded vector spaces. The following identities are satisfied
  \begin{align*}
    \dim_t (V\otimes W) &= \dim_tV\cdot \dim_tW \\
    \dim_t V^* &= \dim_{t^{-1}}V \\
    \dim_t V^{(1)}  &= \dim_{t^p}(V). \\
  \end{align*}
\end{prop}
\begin{proof}
  The characters  satisfy
  analogous properties (cf. \cite[I.2]{jan2003}) to those sketched above. Thus we can apply $\varphi$ to 
  $\text{ch}\,V, \text{ch}\,W, \text{ch}\,V^*$ etc. to get
  the above identities. 
  For example, the last identity follows from the fact that
  $\text{wht}(p\lambda) = p\,\text{wht}(\lambda)$ and hence
  $\varphi(e(p\lambda)) = (t^p)^{-2\,\text{wht}(\lambda)}$. 
\end{proof}
We define the twisting endomorphism 
\begin{align*}
  (-)^F: \mathbb{Z}[t,t^{-1}] &\rightarrow \mathbb{Z}[t,t^{-1}]\\ 
               t &\mapsto t^p
\end{align*}
which is compatible, under $\varphi$, with the endomorphism $(-)^F$
on $\mathbb{Z}[X(T)]$. 

The following lemma regarding the derivative of the generic dimension will
prove useful when we calculate $V_{G_1}(H^i(\lambda))$ for  $\lambda$ subregular. 
\begin{lem}\label{lem:deriv-vanish}
  Let $G$ be a semisimple simply connected algebraic group over $k$, and let
  $\theta \in \mathbb{Z}[X(T)]^W$. If $f(t) = \varphi(\theta) \in \mathbb{Z}[t,t^{-1}] $ denotes the specialization of $\theta$ to $\mathbb{Z}[t,t^{-1}]$, then $f'(1) = 0$. 
\end{lem}
\begin{proof}
For each $\lambda \in X(T)$, we define the orbit sum 
\[
  s(\lambda) = \sum_{w\in W}e(w\lambda) \in \mathbb{Z}[X(T)]^W.
\]
The set of all $s(\lambda)$ with $\lambda \in X(T)_+$ form a basis for $\mathbb{Z}[X(T)]^W$. 
For any $W$ invariant element $\theta$, we can always write  
$\theta = \sum_{\lambda \in X(T)_+}a_{\lambda}s(\lambda)$. 
Therefore, we can assume without loss of generality that $\theta = s(\lambda)$
for some $\lambda \in X(T)_+$. Thus
\[
  f(t) = \sum_{w\in W}t^{-2\,\text{wht}(w\lambda)}
\]
and 
\[
  f'(1) = -2\sum_{w\in W}\text{wht}(w\lambda)
\]
where if we write $\lambda = \sum_{i=1}^nn_i\alpha_i$, then $\text{wht}(\lambda) = \sum n_id_i$,
in particular, $\text{wht}(\alpha_i) = d_i$.  
Since the $\text{wht}$ map is linear  for each $w \in W$ 
\[
  \text{wht}(w\lambda) = \sum_{i=1}^nn_i\text{wht}(w\alpha_i).
\]
This gives us 
\[
  f'(1) = -2\sum_{w\in W}\sum_{i=1}^nn_i\text{wht}(w\alpha_i).
\]

It now suffices to show that for any simple root $\alpha$,
\[
  \sum_{w\in W}\text{wht}(w\alpha) = 0.
\]
If we let $W_{\alpha}^+ = \{ w \in W \mid w\alpha \in \Phi^+\}$ and $W_{\alpha}^- = \{ w \in W \mid w\alpha \in \Phi^-\}$, then 
\[
W = W_{\alpha}^+ \sqcup W_{\alpha}^-.
\]
Also, for each $\beta \in \Phi^+$, let $s_{\beta} \in W$ denote the reflection across the hyperplane orthogonal to $\beta$. 
We can see that for $w \in W_{\alpha}^+$, $(s_{w\alpha}w)\alpha = s_{w\alpha}(w\alpha) = -w\alpha$.
Thus, the map
\begin{align*}
W_{\alpha}^+ &\rightarrow W_{\alpha}^- \\
w &\mapsto s_{w\alpha}w
\end{align*}
is a bijection, whose inverse is given by 
\begin{align*}
W_{\alpha}^- &\rightarrow W_{\alpha}^+ \\
w &\mapsto s_{w\alpha}w.
\end{align*}
Therefore, 
\begin{align*}
  \sum_{w\in W}\text{wht}(w\alpha)  &=  \sum_{w\in W_{\alpha}^+}\text{wht}(w\alpha) +  \sum_{w'\in W_{\alpha}^-}\text{wht}(w'\alpha) \\
  						      &= \sum_{w\in W_{\alpha}^+}(\text{wht}(w\alpha) +  \text{wht}((s_{w\alpha}w)\alpha)) \\
						      &= \sum_{w\in W_{\alpha}^+}(\text{wht}(w\alpha) + \text{wht}(-w\alpha)) \\
						      &= 0.
\end{align*}

\end{proof}
The following two corollaries are immediate. 
\begin{cor}
  Let $G$ be a semisimple algebraic group over $k$, and let 
  $\theta^F \in \mathbb{Z}[X(T)]^W$ be arbitrary.
If $f(t) = \varphi(\theta^F) \in \mathbb{Z}[t,t^{-1}] $ denotes the specialization of $\theta^F$, 
  then for any $p^{th}$ root of unity $\zeta$, $f'(\zeta)=0$. 
\end{cor}
\begin{proof}
  If we let $g(t) = \varphi(\theta)$ so that$f(t) = g(t^p)$, then by the chain rule
  $f'(\zeta) = p\zeta^{p-1}g'(1)=0$.
\end{proof}
\begin{cor}
  For any $G$-module $M$ with Frobenius twist $M^{(1)}$,
  \[ 
    (\dim_t M^{(1)})'(\zeta) = 0.
  \]
\end{cor}
\subsection{Quantum dimension}
The notion of the quantum dimension of a module
was first introduced by Parshall-Wang (cf. \cite[Section 2]{PW1993}), where they applied it to the theory of 
support varieties for quantum groups.

For any  primitive $p^{\text{th}}$ root of unity $\zeta$, let 
\begin{align*}
  \text{ev}_{\zeta}: Z[t,t^{-1}] &\rightarrow \mathbb{Z}[\zeta] \\
                     f(t)        &\mapsto     f(\zeta) 
\end{align*}
denote the evaluation map.
 Then $\text{ev}_{\zeta}$ induces an isomorphism
\[
  \frac{\mathbb{Z}[t,t^{-1}]}{\psi_p(t)\mathbb{Z}[t,t^{-1}]} \cong 
  \mathbb{Z}[\zeta],
\]
where $\psi_p(t) = 1 + t \cdots + t^{p-1}$ is the $p^{th}$ cyclotomic polynomial.

For any $T$-module $V$,  the \emph{quantum dimension} of $V$ is the
evaluation
\[
  \dim_{\zeta}V = \text{ev}_{\zeta}(\dim_t V)
\]

One can see that all of the properties for the generic dimension defined earlier carry over to the quantum dimension. 
In particular, for any $T$-module $V$
\[
  \dim_{\zeta} V^{(1)} = \dim_{\zeta^p}V = \dim_1 V = \dim V \in \mathbb{Z}
\]
so the quantum dimension of a twisted module $V^{(1)}$ is the same as its dimension as
a vector space. 
In general, for any $f(t) \in \mathbb{Z}[t,t^{-1}]$, $\text{ev}_{\zeta}(f(t)^F) = f(1)$ and
hence $\text{ev}_1= \text{ev}_{\zeta} \circ (-)^F.$

\section{Support varieties}\label{sec:supp_var}
\subsection{ }
Let $H$ be a finite group scheme over $k$. We can form the \emph{cohomology variety} 
$V_H=\text{max}(A)$ where
\[
  A = \begin{cases} 
    \bigoplus_{i=0}^{\infty}\text{Ext}_H^{i}(k,k) &\mbox{if $p$ is even}\\ 
    \bigoplus_{i=0}^{\infty}\text{Ext}_H^{2i}(k,k) &\mbox{if $p$ is odd}. 
  \end{cases}
\]
It is well known that $A$ is a commutative and finitely generated algebra (cf. \cite[Theorem 1.1]{fs}), and
hence $V_H$ is an affine $k$-variety. Now for any $H$-module $M$, the ring
\[
  \bigoplus_{i=0}^{\infty}\text{Ext}_H^i(M,M) 
\]
is naturally an $A$-module via the Yoneda action.  
The \emph{support variety} 
\[
V_H(M) := V(\text{ann}_A(\text{Ext}^*_{H}(M,M))) \subseteq V_H
\]
 is the zero locus of the annihilator
ideal. 
Support varieties satisfy a number of useful identities. 
\begin{prop}\label{prop:basic_properties}
  Let $M_1$, $M_2$, $M_3$ be arbitrary $H$-modules then
  \begin{enumerate}
    \item[(a)] $V_H(M_1 \oplus M_2) = V_H(M_1)\cup V_H(M_2)$
    \item[(b)] $V_H(M_1\otimes M_2) = V_H(M_1)\cap V_H(M_2)$
    \item[(c)] If $0 \rightarrow M_1 \rightarrow M_2 \rightarrow M_3 \rightarrow 0$ is exact then 
    \[ V_H(M_{i}) \subseteq V_H(M_j)\cup V_H(M_k)
    \]
    where $\{i,j,k\} = \{1,2,3 \}$. 
  \end{enumerate}
\end{prop}
\begin{proof}
  See \cite[Theorem 5.6]{fp2005}.
\end{proof}

\subsection{Computing the dimensions of support varieties}
If $H$ is a finite group scheme over $k$ and $M$ is any $H$-module, we define the \emph{complexity}
$c_H(M)$ of $M$ to be the smallest integer $d$ such that if
\[
  \cdots \rightarrow P_2\rightarrow P_1 \rightarrow P_0 \rightarrow M \rightarrow 0
\]
is a minimal projective resolution of $M$, then 
 there exists $C>0$ satisfying
\[
  \dim_k P_n \leq C\cdot n^{d-1}
\]
for all $n \geq 1$. 
It is known that $c_H(M) = \dim V_{H}(M)$, the Krull dimension of $V_H(M)$  (cf. ~\cite[Proposition 5.6(7)]{fp2005}).  

In the case where $H=G_1$, one can relate the complexity of a $G$-module $M$, regarded as a $G_1$-module,
to the divisibility of its generic dimension by $\psi_p(t) = 1 + t + \cdots + t^{p-1}$. 
More precisely, if $p$ is good and $\psi_p(t)^s \nmid \dim_t M $, then by 
\cite[Theorem 3.4.1b]{npv2002}
\[
  c_{G_1}(M) \geq |\Phi| - d(\Phi,p)-2(s-1)
\]
where
\[
  d(\Phi,p) = |\{\alpha \in \Phi\,\mid \,d_{\alpha}\langle\rho,\alpha^{\vee}\rangle \in p\mathbb{Z}\}|.
\]
In particular, if $b=|\Phi|-d(\Phi,p)-c_{G_1}(M)$, then $p^{b/2} \nmid \dim M$. 
This follows by setting $t=1$ and by making the observation that $\dim_1 M =\dim M $ and 
$\psi_p(1)=p$. 
Similar techniques were employed in \cite{uga2007} to deal with the $p$-bad case. 
\subsection{General results on support varieties of $G_1$-modules}
By \cite[Theorem 5.2]{sfb}, the varieties
$V_{G_1}$ and  $\mathcal{N}_1(G)$ are naturally homeomorphic,
where 
\[
  \mathcal{N}_1(G) = \{x \in \mathfrak{g}:=\text{Lie}(G) \, \mid \, x^{[p]} = 0 \}
\]
is the $p$-restricted nullcone.  
 It is also known that for $p\geq h$, $\mathcal{N}_1(G) = \mathcal{N}(G)$, where $\mathcal{N}(G)$ is the nilpotent cone for $G$. Moreover, if $p>h$, then 
 $V_{G_1}$ and $\mathcal{N}(G)$ are naturally isomorphic as varieties (see \cite[Corollary 3.7]{aj}).
For simplicity, we will often denote $\mathcal{N}_1 := \mathcal{N}_1(G)$ and $\mathcal{N} :=\mathcal{N}(G)$.

For each subset $I \subseteq \Pi$, let  $U_I \subseteq G$ be the unipotent subgroup generated by the roots
 $\alpha \in (-\Phi^+)\backslash \Phi_I$
as in~\cite[II.1.8]{jan2003}.  We denote the corresponding Lie algebra by $\mathfrak{u}_I = \text{Lie}(U_I)$. 
It follows that 
$G\cdot \mathfrak{u}_I \subseteq \mathcal{N}$ is a closed subvariety whose dimension is given by
\[
  \dim G\cdot \mathfrak{u}_I = |\Phi|-|\Phi_I|.
\]

Now fix any $\lambda \in X(T)_+$ with $w(\Phi_{\lambda})=\Phi_I$ for some $w \in W$,
then it was shown in~\cite[Section 7.4.1]{npv2002} that
\[
  V_{G_1}(L(\lambda)) \subseteq G\cdot \mathfrak{u}_I.
\]
More generally, for any $y \in W_p$ with $y\cdot \lambda \in X(T)_+$
\[
  V_{G_1}(L(y\cdot \lambda)) \subseteq G\cdot \mathfrak{u}_I.
\]
Due to the lack of a suitable reference, the following lemma has been included. 
\begin{lem}\label{lem:linkage}
 If $M$ is a finite dimensional $G$-module such that every composition factor is of the form $L(y\cdot \lambda)$  
for some $y \in W_p$, then
$V_{G_1}(M) \subseteq G\cdot \mathfrak{u}_I$ where $w(\Phi_{\lambda}) = \Phi_I$ for some $w \in W$. 
\end{lem}
\begin{proof}
There exists a filtration 
\[
0 = M_0 \subsetneq M_1 \subsetneq \dots \subsetneq M_r = M 
\]
with $M_i/M_{i-1} \cong L(y_i\cdot \lambda)$ for $i=1,\dots, r$. 
This gives us exact sequences 
\[
0 \rightarrow M_{i-1} \rightarrow M_i \rightarrow L(y_i\cdot \lambda) \rightarrow 0,
\]
thus by Proposition ~\ref{prop:basic_properties} and the remark preceding this lemma 
\[
  V_{G_1}(M_i) \subseteq V_{G_1}(M_{i-1})\cup G\cdot \mathfrak{u}_I.
  \]
 The result now follows from induction on $i$. 
 
\end{proof}
Next we state a result which generalizes \cite[Theorem 6.2.1]{npv2002}.
\begin{prop}\label{prop:npv_result}
  Let $\lambda \in X(T)$ be any weight with $w(\Phi_{\lambda})=\Phi_I$ for some $w \in W$,
  and suppose $\bigoplus_{i\geq 0}H^i(\lambda)\neq 0$, then there exists $j\geq 0$ such that 
  \[
    V_{G_1}(H^j(\lambda))=G\cdot \mathfrak{u}_I.
  \]
\end{prop}
\begin{proof}
  By the Strong Linkage Principle (cf. \cite[Proposition II.6.13]{jan2003}), every composition factor of $H^i(\lambda)$ 
  is of the form $L(y\cdot \lambda)$ where $y\cdot \lambda \in X(T)_+$ and $y \in W_p$. 
 It follows from Lemma~\ref{lem:linkage} that 
 \[
 V_{G_1}(H^i(\lambda))\subseteq G\cdot \mathfrak{u}_I
 \]
 for each $i\geq 0$.   
  The (generic) Euler characteristic
  \[
    D_t(\lambda) = \sum_{i\geq 0}(-1)^i\text{dim}_t H^i(\lambda)
  \]
  is given by Weyl's generic dimension formula.  
  One can then verify that $s=\frac{1}{2}(|\Phi_I|-d(\Phi,p))$ is the multiplicity
  of $\psi_p(t)$ as a divisor of $D_t(\lambda)$. 

  Now suppose that for all $i$,   $V_{G_1}(H^i(\lambda))\subsetneq G\cdot \mathfrak{u}_I$.
  This would imply that $\psi_p(t)^{s+1} \mid \dim_t H^i(\lambda)$ for all $i$  and hence
  $\psi_p(t)^{s+1} \mid D_t(\lambda)$, a contradiction.
\end{proof}

It may also be useful to mention that in \cite[Theorem 3.3]{DNP} the support varieties $V_{G_1}(L(\lambda))$
were determined for all the simple modules $L(\lambda)$ whenever the Lusztig conjecture holds.
Recently some work has been done in studying the varieties  $V_{G_r}(L(\lambda))$ for $\lambda \in X_r(T)$
(cf. \cite{Sobaje}).

\section{Character formulas for line bundle cohomology groups when $G=SL_3(k)$}\label{sec:dim_formulas}
\subsection{ }
From now on we shall assume that $G=SL_3(k)$. 
In this section we will present recursive character formulas for the sheaf cohomology groups
$H^i(\lambda)$ where  $\lambda \in X(T)$ is any weight.
If we apply the map
\begin{align*}
\varphi: \mathbb{Z}[X(T)] &\rightarrow \mathbb{Z}[t,t^{-1}]
\end{align*}
defined in Section~\ref{sec:gen_dim}
to the character formulas,  we will also get recursive expressions for the generic dimensions of the modules $H^i(\lambda)$. 
The generic dimension formulas will be used in 
Sections~\ref{sec:reg_case},~\ref{sec:sreg_case} to determine
the multiplicity of $\psi_p(t)$ as a factor of $\dim_tH^i(\lambda)$.
Therefore, by the results of Section~\ref{sec:supp_var}, we will be able to compute
$\dim V_{G_1}(H^i(\lambda))$. 
  
  We begin by  introducing some additional notation.  
Let $\Phi$ be the root system of type $A_2$ with basis $\Pi=\{\alpha,\beta\}$ and let
$\omega_{\alpha}$, $\omega_{\beta}$ be the corresponding fundamental weights.
The weights in $X(T) \cong \mathbb{Z}^2$ will be given by coordinates
 $(r,s) = r\omega_{\alpha} + s\omega_{\beta}$, for example, $\alpha = (2,-1)$ and $\beta = (-1,2)$.
Finally, we define the \emph{fundamental line} to be the collection of all weights of the form $(r,-r-1)$ for $r \in \mathbb{Z}$. 
In other words, it is unique line in $X(T)$ which  passes through the points $w_{\alpha}-\rho = (0,-1)$ and $w_{\beta} - \rho = (-1,0)$. 

Let $N(\alpha)$ and $N(\beta)$ be the two unique indecomposable $B$-modules such that 
\begin{align*}
  \text{ch}\,N(\alpha) &= e(-\alpha) + e(0) \\
  \text{ch}\,N(\beta) &= e(-\beta) + e(0).
\end{align*}
For any weight $(r,s) \in X(T)$, let 
\begin{align*}
  \chi_p(r,s) &= \text{ch}\,L(r,s)\\
  \chi^i(r,s) &= \text{ch}\, H^i(r,s) \\
  \chi^i_{\alpha}(r,s) &= \text{ch}\,H^i(N(\alpha)\otimes (r,s)) \\
  \chi^i_{\beta}(r,s) &= \text{ch}\,H^i(N(\beta)\otimes (r,s)).
\end{align*}
The ordinary dimensions are denoted by
\begin{align*}
  \delta_p(r,s) &= \dim L(r,s)\\
  \delta^i(r,s) &= \dim H^i(r,s) \\
  \delta^i_{\alpha}(r,s) &= \dim H^i(N(\alpha)\otimes (r,s)) \\
  \delta^i_{\beta}(r,s) &= \dim H^i(N(\beta)\otimes (r,s)).
\end{align*}
Similarly, we denote the generic dimensions by
\begin{align*}
  D_p(r,s) &= \dim_t L(r,s)\\
  D^i_t(r,s) &= \dim_{t}H^i(r,s) \\
  D^i_{\alpha,t}(r,s) &= \dim_{t} H^i(N(\alpha)\otimes (r,s)) \\
  D^i_{\beta,t}(r,s) &= \dim_{t} H^i(N(\beta)\otimes (r,s)).
\end{align*}

In \cite[Lemma 2.1]{DONK2007}, expansions of $\chi^i(r,s)$, $\chi_{\alpha}^i(r,s)$ and $\chi_{\beta}^i(r,s)$ are expressed in terms of the infinitesimal invariants of certain modules, which are computed in  \cite[Sections 2.2-2.5]{DONK2007}. With these calculations in hand, one can immediately derive recursive expansion formulas, which we will present in the next subsection. 
\begin{rem}
Recursive formulas are also given in \cite[Sections 4-6]{DONK2007}, but these formulas contain numerous typographical errors, and fail to agree with the calculations in \cite[Sections 2.2-2.5]{DONK2007}. The formulas given below were obtained by imitating Donkin's proofs, but with greater care in avoiding typos. The references to the corresponding formulas in \cite{DONK2007} are also stated. 
\end{rem}

%

\subsection{ }
We begin with the case when $\lambda = (a+pr, b+ps)$ is a regular weight.
The following proposition is a combination of \cite[Lemma 5.5 and Proposition 6.2]{DONK2007}.
\begin{prop}\label{prop:reg-weight}
    Let $p \geq 3$, and let $\lambda = (a+pr,b+ps)$, with $(a,b) \in X_1(T)$, be an arbitrary regular weight, then 
    if $a + b <p-2$
    \begin{align*}
    \chi^i(a+pr, b + ps) &= \chi_p(a,b)\chi^i(r,s)^F + [\chi_p(p-2-b,p-2-a)+\chi_p(a,b)]\chi^i(r-1,s-1)^F \\
                                  &\quad + \chi_p(a+b+1,p-2-b)\chi^i(r,s-1)^F + \chi_p(p-2-a,a+b+1)\chi^i(r-1,s)^F \\
                                  &\quad + \chi_p(b,p-3-a-b)\chi^i_{\alpha}(r,s-1)^F + 
                                          \chi_p(p-3-a-b,a)\chi^i_{\beta}(r-1,s)^F.
    \end{align*}
     If $a + b > p-2$, then 
    \begin{align*}
    \chi^i(a+pr, b + ps) &= [\chi_p(a,b)+\chi_p(p-2-b,p-2-a)]\chi^i(r,s)^F \\ 
    				             & \quad+ \chi_p(p-2-b,p-2-a)\chi^i(r-1,s-1)^F \\
                          &        \quad + \chi_p(2p-3-a-b,a)\chi^i(r,s-1)^F + \chi_p(b,2p-3-a-b)\chi^i(r-1,s)^F \\
                          &        \quad + \chi_p(a+b-p+1,p-2-b)\chi^i_{\alpha}(r+1,s-1)^F\\
                          &        \quad + \chi_p(p-2-a,a+b-p+1)\chi^i_{\beta}(r-1,s+1)^F.
     \end{align*}
   \end{prop}

Likewise, we get recursive formulas for the $N(\alpha)\otimes \lambda$ bundles when $\lambda$ is regular
(cf. \cite[Lemma 5.2.6 and Lemma 6.4.3]{DONK2007}).
\begin{prop}
  Let $p \geq 3$  and $(a,b) \in X_1(T)$, then if $a \neq 0$
  \begin{equation*}
    \chi^i_{\alpha}(a + pr, b +ps) = \chi^i(a+pr, b+pr) + \chi^i(a-2 +pr, b+1+ps)
  \end{equation*}
  otherwise, when $a=0$ 
  \begin{align*}  
    \chi^i_{\alpha}(pr, b + ps) &= 2\chi_p(p-2-b,p-2)\chi^i(r-1,s-1)^F + 2\chi_p(b,p-3-b)\chi^i_{\alpha}(r,s-1)^F \\
                                 & \quad + \chi_p(p-3-b,0)[\chi^i(r-1,s)+\chi(0,1)\chi^i(r-1,s-1)]^F\\
                                 & \quad + 2\chi_p(p-2,b+1)\chi^i(r-1,s)^F \\
				            &\quad + \chi_p(b+1,p-2-b)\chi^i_{\alpha}(r,s-1)^F + \chi_p(0,b)[\chi^i(r-1,s-1)+\chi^i(r-1,s)]^F.
  \end{align*}
\end{prop}
Now we consider the case where $\lambda$ is subregular. The following is a combination of
\cite[Lemma 4.2.1, Lemma 5.3.5 and Lemma 6.1.3]{DONK2007}.
\begin{prop}\label{prop:sreg_formulas}
  Let $p \geq 2$ and $0 \leq a < p-1$. Then 
 \begin{equation*}
   \begin{aligned}
    \chi^i(p-1+pr, a + ps) &=  \chi(p-1,a)\chi^i(r,s)^F + \chi(p-2-a,p-1)\chi^i(r,s-1)^F \\
                                  & \quad + \chi(a,p-2-a)\chi^i_{\alpha}(r+1,s-1)^F \\
    \chi^i(p-2-a + pr, p-1 + ps) &= \chi(p-2-a,p-1)\chi^i(r,s)^F + \chi(p-1,a)\chi^i(r-1,s)^F + \\
                                  &\quad + \chi^i(a,p-2-a)\chi^i_{\beta}(r-1,s+1)^F \\
    \chi^i(a+pr, p-2-a + ps) &= \chi(a,p-2-a)\chi^i(r,s)^F + \chi(p-1,a)\chi^i(r,s-1)^F \\
   			 &\quad+ \chi(p-2-a,p-1)\chi^i(r-1,s)^F + \chi(a,p-2-a)\chi^i(r-1,s-1)^F.
   \end{aligned}
   \end{equation*}
 \end{prop}

 The character formulas for the $N(\alpha)\otimes\lambda$ bundles when $\lambda$ is subregular are now given
 (cf.  \cite[Lemma 4.2.2, Lemma 5.3.6 and Lemma 6.3.1]{DONK2007}).
 However,  in \cite[Lemma 6.3.1]{DONK2007}, the formula for $\chi^i_{\alpha}(pr,p-1+ps)$ was
 accidentally omitted. 
 \begin{prop}
   Let $p \geq 2$. For $0 \leq a < p-1$,
   \begin{equation*}
     \chi^i_{\alpha}(p-1+pr, a+ps) = \chi^i(p-1 +pr, a +ps) + \chi^i(p-3 +pr, a+1 +ps).
   \end{equation*}
   For $0 \leq a < p-2$,
   \begin{equation*}
     \chi^i_{\alpha}(p-2-a+pr, p-1+ps) = \chi^i(p-2-a+pr,p-1+ps) + \chi^i(p-4-a +pr, p(s+1))
   \end{equation*}
   and when $a = p-2$,
   \begin{align*}
     \chi^i_{\alpha}(pr, p-1 + ps) &= 2\chi(p-1,p-2)\chi^i(r-1,s)^F + \chi(p-2,0)[\chi^i(r-1,s+1)+\chi(0,1)\chi^i(r-1,s)]^F\\
                                  &\quad + \chi(0,p-1)\chi^i_{\alpha}(r,s)^F. 
   \end{align*}
   For $0 <a \leq p-2$,
   \begin{equation*}
     \chi^i_{\alpha}(a+pr, p-2-a+ps) = \chi^i(a+pr,p-2-a+ps) + \chi^i(a-2 +pr,p-1-a +ps)
   \end{equation*}
   and when $a=0$,
   \begin{align*}
      \chi^i_{\alpha}(pr, p-2 + ps) &= \chi(0,p-2)[\chi(1,0)\chi^i(r-1,s) + \chi^i(r-1,s-1)]^F + \chi(p-1,0)\chi^i_{\alpha}(r,s-1)^F \\
                                  &\quad + 2\chi(p-2,p-1)\chi^i(r-1,s)^F. 
    \end{align*}
\end{prop}

Finally, by the Andersen-Haboush identity (cf. \cite[Proposition II.3.19]{jan2003}),
we can consider the case when $\lambda$ is in the Steinberg block.
\begin{prop}
  For $p \geq 2$,
  \[
    \chi^i(p-1 +pr, p-1 +ps) = p^3 \chi^i(r,s) 
  \]
  and 
  \[
    \chi^i_{\alpha}(p-1+pr, p-1+ps) = \chi^i(p-1 +pr, p-1+ps) + \chi^i(p-3 +pr, p(s+1)). 
  \]
\end{prop}

\subsection{ }\label{sec:transposition}
  For any $(r,s) \in X(T)$,
  \[
    \chi^i_{\beta}(r, s) = \chi^i_{\alpha}(s,r)^{\tau}
  \]
  where the ring automorphism $(\cdot)^{\tau}: \mathbb{Z}[X(T)] \rightarrow \mathbb{Z}[X(T)]$  given by 
  $e(r,s) \mapsto e(s,r)$ is induced by the involutary automorphism $\tau: G\rightarrow G$ 
  (cf. \cite[2.5]{DONK2007} or $\sigma$ in \cite[II.2.13]{jan2003} for a more general discussion). 
  By applying $\tau$ to the above formulas for the $N(\alpha)$-bundles we get
  recursive expansion formulas for $\chi^i_{\beta}(\lambda)$ for all $\lambda \in X(T)$.

\section{Calculations for regular weights}\label{sec:reg_case}
\subsection{ }
In this section we will show that if $\lambda \in X(T)$ is any regular weight such that for some $i \geq 0$
$H^i(\lambda) \neq 0$, then $V_{G_1}(H^i(\lambda)) = \mathcal{N}$, the nilpotent cone. 
The case in which $\mathcal{L}(\lambda)$  has just one non-vanishing cohomology group  
follows immediately from Proposition~\ref{prop:npv_result}.
To deal with the case in which $\mathcal{L}(\lambda)$ has multiple non-vanishing cohomology groups, we will need to employ
the recursive character formulas given in Section~\ref{sec:dim_formulas}.
They will be used to show that the quantum dimension 
$\dim_{\zeta}H^i(\lambda) \neq 0$. Together with the results presented in Section~\ref{sec:supp_var}, it will follow that
\[
\dim V_{G_1}(H^i(\lambda)) = |\Phi| = \dim \mathcal{N}.
\]
 In addition, since $\mathcal{N}=G\cdot \mathfrak{u}$ is an irreducible variety, the non-vanishing of the quantum dimension will imply that
 $V_{G_1}(H^i(\lambda))=\mathcal{N}$. 
Therefore, the rest of this section will be devoted to showing that 
\[
D^i_{\zeta}(r,s) := \dim_{\zeta}(H^i(r,s)) \neq 0
\] 
whenever $(r,s) \in X(T)$ is regular and has multiple non-vanishing cohomology groups. 

By specializing the formulas given in Proposition~\ref{prop:reg-weight}  to the quantum dimension,
we get the following proposition. 
\begin{prop}\label{prop:general-char-nonrecursive}
  For any $p \geq 3$ and $(a,b) \in X_1(T)$ with $a + b < p-2$,
\begin{align*}
  D^i_{\zeta}(a+pr,b+ps) &= D^0_{\zeta}(a,b)[\delta^i(r,s) -\delta^i(r-1,s-1) -2\delta^i(r,s-1) - 2\delta^i(r-1,s) \\
  & +\delta^i_{\alpha}(r,s-1) + \delta^i_{\beta}(r-1,s)].
 \end{align*}
 If $a+b > p-2$,
 \begin{align*}
   D^i_{\zeta}(a+pr,b+ps) &= -D^0_{\zeta}(a,b)[-\delta^i(r,s) +\delta^i(r-1,s-1) -2\delta^i(r,s-1) - 2\delta^i(r-1,s) \\ 
  & +\delta^i_{\alpha}(r+1,s-1) + \delta^i_{\beta}(r-1,s+1)].
 \end{align*}
\end{prop}
As a result, we can see that the quantum dimensions are really determined by the integers
\begin{equation}\label{eq: ST-nonrecursive}
\begin{aligned}
S^i(r,s) &= \delta^i(r,s)-\delta^i(r-1,s-1)-2\delta^i(r,s-1)-2\delta^i(r-1,s)\\
         &+\delta^i_{\alpha}(r,s-1)+\delta^i_{\beta}(r-1,s) \\
T^i(r,s) &= -\delta^i(r,s)+\delta^i(r-1,s-1)-2\delta^i(r,s-1)-2\delta^i(r-1,s)\\
         &+\delta^i_{\alpha}(r+1,s-1)+\delta^i_{\beta}(r-1,s+1).
\end{aligned}
\end{equation}
One can verify that  $D^0_{\zeta}(0,0) = 1$ and $D^0_{\zeta}(p-2,p-2) = -1$, thus,
for any $(r,s) \in X(T)$
\begin{equation*}
\begin{aligned}
S^i(r,s) &= D^i(pr,ps) \\
T^i(r,s) &= D^i(p-2+pr,p-2+ps).
\end{aligned}
\end{equation*}
Therefore, by calculating the integers $S^i(r,s)$ and $T^i(r,s)$, we will determine
 $D^i_{\zeta}(a+pr,b+ps)$ for all $(a,b) \in X_1(T)$.
It will be sufficient to calculate the integers $S^i(r,s)$ and $T^i(r,s)$.  
\begin{rem}
Figure~\ref{pic:p=5} graphically depicts certain values of $S^1(r,s)$ and $T^1(r,s)$ when $p=5$. The apparent symmetry of the numbers appearing in the figure suggests that $S^1(r,s)$ and $T^1(r,s)$ may be given by simpler recursive formulas.  
\end{rem}


\begin{figure}\label{pic:p=5}
\includegraphics[scale=0.7]{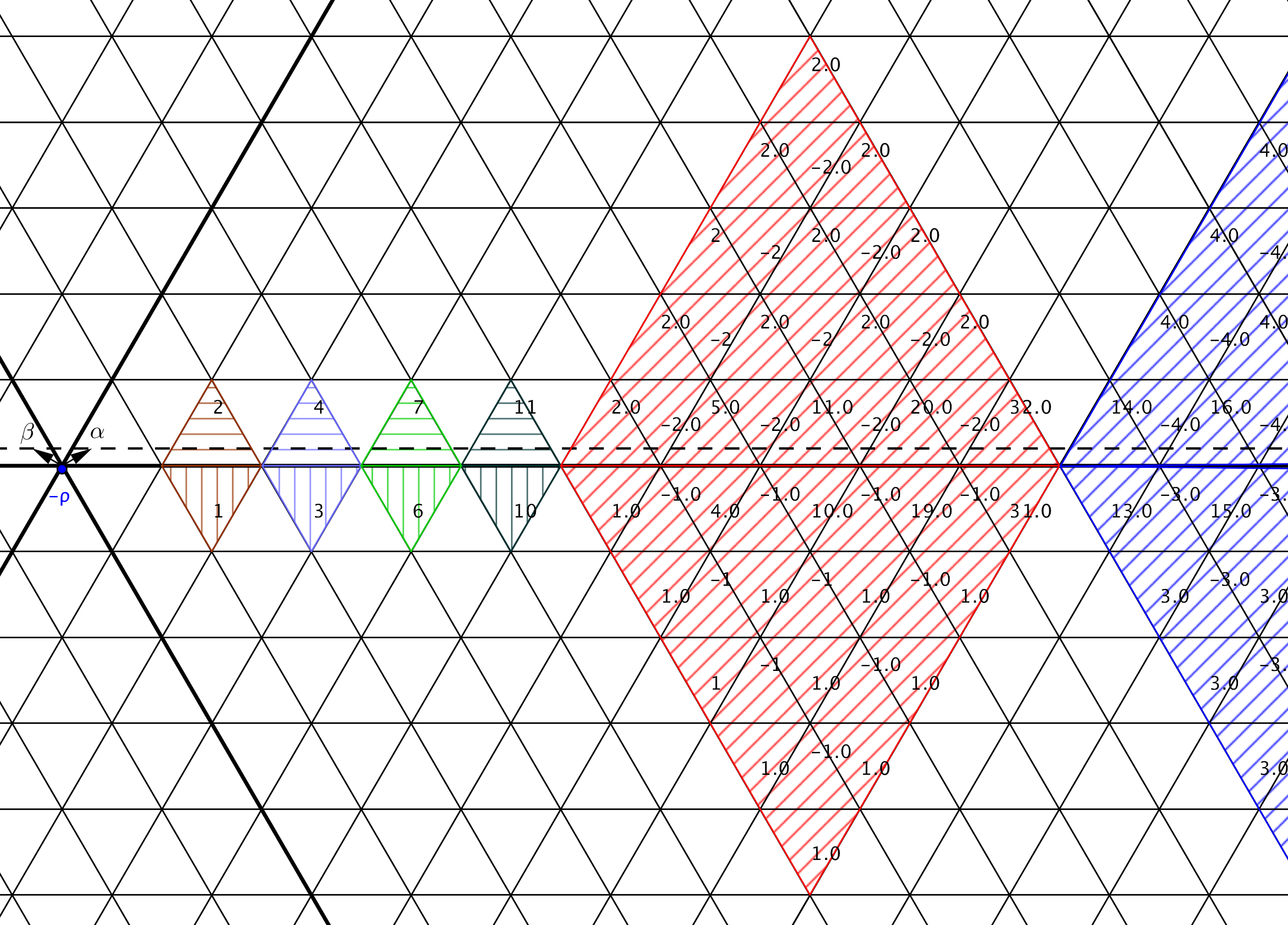}
\caption{This is a picture of the weight lattice for $SL_3(k)$ when $p=5$. The shaded regions consist of the weights that have multiple non-vanishing cohomology groups 
(see also \cite[Fig.1.]{and21979}). For certain weights $(r,s) \in X(T)$, the value of $S^1(r,s)$ has been placed at the point $(5r,5s)$ and the value of $T^1(r,s)$ has been placed 
at the point $(3+5r, 3+5s)$. For instance, the numbers $S^1(3,-4) = 6$ and $T^1(3,-4) = 7$ are labeled at the points $(15,-20)$ and $(18,-17)$ (the bottom and top of the green ``diamond'') respectively.  }
 \end{figure}

\subsection{}
Our strategy begins by writing  $(r,s) = (x+pr_0,y+ps_0)$ for some $(x,y) \in X_1(T)$,  which gives us
\begin{equation*}
\begin{aligned}
S^i(r,s) &= \delta^i(x+pr_0,y+ps_0)-\delta^i(x-1+pr_0,y-1 +ps_0)-2\delta^i(x+pr_0,y-1+ps_0) \\
  &-2\delta^i(x-1+pr_0,y+ps_0)     +\delta^i_{\alpha}(x+pr_0,y-1+ps_0)+\delta^i_{\beta}(x-1+pr_0,y+ps_0) \\
T^i(r,s) &= -\delta^i(x+pr_0,y+ps_0)+\delta^i(x-1+pr_0,y-1+ps_0)-2\delta^i(x+pr_0,y-1+ps_0)\\
  &-2\delta^i(x-1+pr_0,y+ps_0) +\delta^i_{\alpha}(x+1+pr_0,y-1+ps_0)+\delta^i_{\beta}(x-1+pr_0,y+1+ps_0).
\end{aligned}
\end{equation*}
We then replace each of the  $\delta^i$ terms above
with the expansion formulas from Section~\ref{sec:dim_formulas} so that the formulas for $S^i(r,s)$ and  $T^i(r,s)$ 
will only involve weights near $(r_0,s_0)$. By grouping the terms together, we will build  simpler
expressions.

Observe that for each $(x,y) \in X_1(T)$, the choice of the formulas to substitute into (\ref{eq: ST-nonrecursive})
will vary. So the first step is to explicitly divide $X_1(T)$ into several subsets which 
will cover all the different types of possible substitutions. We then enumerate across every such subset to check all possibilities. 

\subsection{ }
The next proposition will give recursive formulas for $S^i(r,s)$ and $T^i(r,s)$ when $$(r,s)=(x+pr_0,y+ps_0)$$ and $x+y <p-1$.
\begin{prop}\label{prop:bottom_alcove}
  For any weight $(r,s) = (x +pr_0,y+ps_0)$  with $0\leq x,y$,  $x+y < p-1$ and $(r_0,s_0) \in X(T)$
  \begin{align*}
    S^i(r,s) &= S^i(r_0,s_0) \\
    T^i(r, s)  &= -S^i(r_0,s_0).
  \end{align*}
\end{prop}
\begin{proof}
The technique that we will use to verify the above identities is described in the remarks preceding this proposition. 
Namely, for each $(x,y) \in X_1(T)$ with $x+y < p-1$, we will have to expand the formulas in (\ref{eq: ST-nonrecursive})
by substituting in the recursive expansions for each  dimension. It can be verified that there are only
 14 cases for $(x,y)$ where in each case we will choose different formulas to substitute into the equations. The cases
 are
 \begin{description}
 \item[Case 1] $(x,y)=(0,0)$,
 \item[Case 2] $(x,y)=(1,0)$,
 \item[Case 3] $(x,y) = (a,0)$ where $2\leq a <p-2$,
 \item[Case 4] $(x,y) = (0,1)$,
 \item[Case 5] $(x,y)= (0,b)$ where $2 \leq b <p-2$,
 \item[Case 6] $(x,y) = (1,1)$,
 \item[Case 7] $(x,y) = (a,1)$ where $2 \leq a < p-3$,
 \item[Case 8] $(x,y) = (1,b)$ where $2 \leq b < p-3$,
 \item[Case 9] $(x,y)= (a,b)$ where $2 \leq a,b$ and $a+b < p-2$,
 \item[Case 10]$(x,y) = (0,p-2)$,
 \item[Case 11]$(x,y) = (1,p-3)$,
 \item[Case 12] $(x,y) = (a,p-2-a)$ where $2 \leq a <p-3$,
 \item[Case 13] $(x,y) = (p-3,1)$,
 \item[Case 14] $(x,y) = (p-2,0)$.
 \end{description}
 
 To better illustrate this method, we will demonstrate how to verify the first case. 
 From above we see that in Case 1 $(x,y)=(0,0)$ and so $(r,s)=(pr_0,ps_0)$,  thus
 \begin{align*}
 \delta^i(r,s) &= \delta^i(pr_0,ps_0) \\
               &=  \delta_p(0,0)\delta^i(r_0,s_0)+[\delta_p(p-2,p-2)+\delta_p(0,0)]\delta^i(r_0-1,s_0-1) \\
               &\,+  \delta_p(1,p-2)\delta^i(r_0,s_0-1) + \delta_p(p-2,1)\delta^i(r_0-1,s_0)\\
               & \,+ \delta_p(0,p-3)\delta^i_{\alpha}(r_0,s_0-1)+\delta_p(p-3,0)\delta^i_{\beta}(r_0-1,s_0)\\ 
 \delta^i(r-1,s-1)&= \delta^i(p-1+ p(r_0-1),p-1+p(s_0-1)) = p^3\delta^i(r_0-1,s_0-1)\\
 \delta^i(r,s-1) &= \delta^i(pr_0,p-1+p(s_0-1))\\
                 &= \delta_p(0,p-1)\delta^i(r_0,s_0-1)+\delta_p(p-1,p-2)\delta^i(r_0-1,s_0-1)\\
                 &\,+\delta_p(p-2,0)\delta^i_{\beta}(r_0-1,s_0)\\
 \delta^i(r-1,s) &= \delta^i(p-1+p(r_0-1),ps_0) \\
                 &= \delta_p(p-1,0)\delta^i(r_0-1,s_0)+\delta_p(p-2,p-1)\delta^i(r_0-1,s_0-1)\\ 
                 &\,+\delta_p(0,p-2)\delta^i_{\alpha}(r_0,s_0-1)\\
 \delta^i_{\alpha}(r,s-1)  &=\delta^i_{\alpha}(pr_0,p-1+p(s_0-1))\\
                            &=2\delta_p(p-1,p-2)\delta^i(r_0-1,s_0-1)+\\
                            &\,\delta_p(p-2,0)[\delta^i(r_0-1,s_0)+\delta_p(0,1)\delta^i(r_0-1,s_0-1)]\\
                            &\,+\delta_p(0,p-1)\delta^i_{\alpha}(r_0,s_0-1)\\       
 \delta^i_{\beta}(r-1,s) &= \delta^i_{\beta}(p-1+p(r_0-1),ps_0) \\
                         &= 2\delta_p(p-1,p-2)\delta^i(r_0-1,s_0-1) \\
                         &\,+ \delta_p(p-2,0)[\delta^i(r_0,s_0-1)+\delta_p(0,1)\delta^i(r_0-1,s_0-1)]\\
                         &\,+\delta_p(0,p-1)\delta^i_{\beta}(r_0-1,s_0)\\
 \delta^i_{\alpha}(r+1,s-1)  &= \delta^i_{\alpha}(1+pr_0,p-1+p(s_0-1)) \\
                             &= \delta^i(1+pr_0,p-1+p(s_0-1)) + \delta^i(p-1+p(r_0-1),ps_0)\\
 \delta^i_{\beta}(r-1,s+1)   &= \delta^i_{\beta}(p-1+p(r_0-1),1+ps_0) \\
                             &= \delta^i(p-1+p(r_0-1),1+ps_0)+\delta^i(pr_0, p-1+ps_0).
 \end{align*}
 Substituting this into (\ref{eq: ST-nonrecursive}) and simplifying 
 we can see that $S^i(r,s)=S^i(r_0,s_0)$ and $T^i(r,s)=-S^i(r_0,s_0)$.
 By repeating this computation for Cases 2-14, one can verify that the same recursive 
 identity holds, proving the identity stated in hypothesis. 
\end{proof}

\subsection{ }
The following proposition gives similar formulas for $S^i(r,s)$ and $T^i(r,s)$ whenever $(r,s) = (x+pr_0,y+ps_0)$ with $x+y >p-1$. 
\begin{prop}\label{prop:top_alcove}
  For any weight $(r,s) = (x +pr_0,y+ps_0)$  with $(r_0,s_0) \in X(T)$, $0\leq x,y$ and $x+y > p-1$.
  \begin{align*}
    S^i(r,s) &= -T^i(r_0,s_0) \\
    T^i(r, s)  &= T^i(r_0,s_0).
  \end{align*}
\end{prop}
\begin{proof}
As in the proof of Proposition~\ref{prop:bottom_alcove}, we can break this up into a number of distinct cases
\begin{description}
 \item[Case 1] $(x,y)=(1,p-1)$,
 \item[Case 2] $(x,y)=(p-1,1)$,
 \item[Case 3] $(x,y) = (2,p-2)$,
 \item[Case 4] $(x,y) = (p-2,2)$,
 \item[Case 5] $(x,y)= (a,p-a)$ where $3 \leq a \leq  p-3$,
 \item[Case 6] $(x,y) = (a,p-1)$ where $3\leq a \leq p-2$,
 \item[Case 7] $(x,y) = (p-1,a)$ where $3 \leq a \leq p-2$,
 \item[Case 8] $(x,y) = (a,p-2)$ where $3 \leq a \leq p-3$,
 \item[Case 9] $(x,y)= (p-2,a)$ where $3 \leq a \leq p-3$,
 \item[Case 10]$(x,y) = (p-2,p-2)$,
 \item[Case 11]$(x,y) = (a,b)$ where $a,b \leq p-3$ and $a+b \geq p+1$,
 \item[Case 12] $(x,y) = (p-1,p-1)$.
 \end{description}
 First, we say that the module  $H^i( \lambda \otimes N(\alpha))$ is \emph{split} if 
 $  \chi^i_{\alpha}(\lambda) = \chi^i(\lambda) + \chi^i(\lambda-\alpha)$.
 
 Now to illustrate how this theorem is proven, lets begin by considering Case 1. In this instance, $(r,s)$, $(r-1,s)$ and $(r-1,s-1)$ are subregular and $H^i((r-1,s+1)\otimes N(\beta))$ is non-split,
 yet in Case 3, $(r-1,s+1)$ and $(r-1,s-1)$ are subregular.  In Case 5,  $(r-1,s-1)$ is subregular. In Case 6, we see that $(r,s)$ 
 and $(r-1,s)$ are subregular and $H^i((r-1,s+1)\otimes N(\beta ))$ is non split. In Case 8, we observe that
 $(r-1,s+1)$ is subregular. In Case 10, 
 we see that$(r-1,s+1)$ and $(r+1,s-1)$ are subregular. In Case 11, all weights are regular and all bundles split and in Case 12,
 $(r,s)$ is in the Steinberg block.
\end{proof}

Now we just have to consider weights of the form $(r,s)=(x+pr_0,y+ps_0)$ where $x+y=p-1$.  
In this case, the new recursive formulas will also involve the terms
\begin{equation}
\begin{aligned}
\phi^i(r,s) &:= -\delta^i(r,s-1)-3\delta^i(r-1,s)+\delta^i_{\alpha}(r,s-1)+\delta^i_{\beta}(r-1,s+1)\\
\psi^i(r,s) &:= -\delta^i(r-1,s)-3\delta^i(r,s-1)+\delta^i_{\beta}(r-1,s)+\delta^i_{\alpha}(r+1,s-1).
\end{aligned}
\end{equation}
The following identities can be verified
 \begin{equation}\label{eq: phi-psi identity}
 \begin{aligned}
 \phi^i(r,s)&=\psi^i(s,r)\\
 \phi^i(r,s)+\psi^i(r,s)&=S^i(r,s)+T^i(r,s).
 \end{aligned}
 \end{equation}
 
Now we can state the following result.
 \begin{prop}\label{prop:fund_line}
 For any weight  $(r,s)=(x+pr_0,y + ps_0)$ with $(r_0,s_0) \in X(T)$,  $0\leq x \leq p-1$ and 
 $y=p-1-x$, 
 \begin{align*}
 S^i(r,s) &= S^i(r_0,s_0) + \frac{1}{2}x(x+1)[\phi^i(r_0,s_0) + \psi^i(r_0,s_0)] + \frac{1}{2}p(p-1-2x)\phi^i(r_0,s_0)\\
          &= S^i(r_0,s_0) + \frac{1}{2}y(y+1)[\phi^i(r_0,s_0) + \psi^i(r_0,s_0)] + \frac{1}{2}p(p-1-2y)\psi^i(r_0,s_0)\\
 T^i(r,s) &= T^i(r_0,s_0) + \frac{1}{2}x(x+1)[\phi^i(r_0,s_0) + \psi^i(r_0,s_0)] + \frac{1}{2}p(p-1-2x)\phi^i(r_0,s_0)\\
          &= T^i(r_0,s_0) + \frac{1}{2}y(y+1)[\phi^i(r_0,s_0) + \psi^i(r_0,s_0)] + \frac{1}{2}p(p-1-2y)\psi^i(r_0,s_0).
 \end{align*}
\end{prop}
\begin{proof}
The proof is similar to that of the previous two propositions. As before we must consider several 
different cases for $0 \leq x \leq p-1$ which determine the formulas that are substituted into
(\ref{eq: ST-nonrecursive}) when $(x,y)=(x,p-1-x)$. In this instance, there are the five distinct cases
\begin{description}
\item[Case 1] $x = 0$,
\item[Case 2] $x = 1$,
\item[Case 3] $2 \leq x \leq p-3$,
\item[Case 4] $x=p-2$,
\item[Case 5] $x =p-1$.
\end{description}
The identities can be verified by checking each case and using (\ref{eq: phi-psi identity}) when necessary.
\end{proof}

\subsection{ }
 Finally, we will present recursive formulas for  $\phi^i(r,s)$ and $\psi^i(r,s)$.  
 \begin{prop}\label{prop:phi_rec}
 Let $(r,s)=(x+pr_0,y+ps_0)$ with $(x,y) \in X_1(T)$ and $(r_0,s_0) \in X(T)$, then 
 if  $(x,y)=(x,p-1-x)$,
 \begin{equation*}
 \begin{aligned}
 \phi^i(r,s)&=\frac{1}{2}[p(p+1-2x) + x(x-1)]\phi^i(r_0,s_0)+\frac{1}{2}x(x-1)\psi^i(r_0,s_0)\\
            &= \frac{1}{2}x(x-1)[\phi^i(r_0,s_0)+\psi^i(r_0,s_0)]+\frac{1}{2}p(p+1-2x)\phi^i(r_0,s_0)\\
 \psi^i(r,s)&=\frac{1}{2}[p(p+1-2(p-1-x))+(p-1-x)(p-2-x)]\psi^i(r_0,s_0) \\
          &\,+\frac{1}{2}(p-1-x)(p-2-x)\phi^i(r_0,s_0)\\
            &= \frac{1}{2}y(y-1)[\phi^i(r_0,s_0)+\psi^i(r_0,s_0)]+\frac{1}{2}p(p+1-2y)\psi^i(r_0,s_0),
 \end{aligned}
 \end{equation*}
 otherwise,
 \[
   \phi^i(r,s)=\psi^i(r,s) =0.
 \]
 \end{prop}
 \begin{proof}
 The proof of this proposition is similar to that of the earlier propositions.
 \end{proof} 
 
 \subsection{ }
Now that we have established our formulas, we can proceed with the proof of the non-vanishing result for the quantum dimension.
\begin{lem}\label{lem: monotonicity}
Let $(r,s) = (x+pr_0,y+ps_0)$. If $\phi^i(r_0,s_0) \geq 0$ and $\psi^i(r_0,s_0) \geq 0$,
then 
\begin{align*}
\phi^i(r,s)&\geq \phi^i(r_0,s_0) \geq 0,\\
 \psi^i(r,s) &\geq \psi^i(r_0,s_0) \geq 0. 
 \end{align*}
\end{lem}
\begin{proof}
If $x+y \neq p-1$, then $\phi^i(r,s)=\psi^i(r,s)=0$ by Proposition~\ref{prop:phi_rec}, so the result is trivial. 
Now assume that $y=p-1-x$, by Proposition~\ref{prop:phi_rec}, we get
\[
\phi^i(r,s) = \frac{1}{2}[p(p+1-2x) + x(x-1)]\phi^i(r_0,s_0)+\frac{1}{2}x(x-1)\psi^i(r_0,s_0).
\]
Thus it is sufficient to show that 
\[
f(x,p)=\frac{1}{2}[p(p+1-2x)+x(x-1)]\geq 0.
\]
This follows from a simple analytic argument:  if we treat $p>0$ as a constant and regard
$f(x,p)$ as a function for all $x \in \mathbb{R}$, then it can be shown that $f(x,p) \geq -1/4$
for all $x \in \mathbb{R}$. In particular, since $f(x,p) \in \mathbb{Z}$ whenever $x\in \mathbb{Z}$, we have that  
 $f(x,p) \geq 0$ for all $x \in \mathbb{Z}$.
\end{proof}

The following two lemmas will allow us to reduce the proof of the non-vanishing result to
the case where $(r,s)$ lies on the fundamental line. 
\begin{lem}\label{lem:constant_formula}
  If $(r,s) = (a+pr_0,b+ps_0)$ with $r_0 + s_0 \neq -1$, then
  \begin{align*}
    T^i(r,s) &= T^i(r_0,s_0) \\
    S^i(r,s) &= S^i(r_0,s_0).
  \end{align*}
\end{lem}  
\begin{proof}
Since $(r_0,s_0)$ is not a fundamental line weight, we know that $\phi^i(r_0,s_0)=0=\psi^i(r_0,s_0)$
and that $S^i(r_0,s_0) = -T^i(r_0,s_0)$. 
Thus if $a+b > p-1$, by Proposition~\ref{prop:bottom_alcove}, we get 
\begin{align*}
  T^i(r,s) &= T^i(r_0,s_0)\\
  S^i(r,s) &= -T^i(r_0,s_0) = S^i(r_0,s_0).
\end{align*}
In addition, since $\phi^i(r_0,s_0) = 0 = \psi^i(r_0,s_0)$, by Proposition~\ref{prop:fund_line}, 
when $a + b = p-1$,
\begin{align*}
    T^i(r,s) &= T^i(r_0,s_0) \\
    S^i(r,s) &= S^i(r_0,s_0).
\end{align*}
Similarly, when $a+b < p-1$, we get 
\begin{align*}
  T^i(r,s) &= -S^i(r_0,s_0)=T^i(r_0,s_0)\\
  S^i(r,s) &=  S^i(r_0,s_0).
\end{align*}
\end{proof}

First, we observe that for any weight $(r,s)$ which satisfies $(r+1)(s+1) \leq 0$ (i.e. if $(r,s)$ is neither dominant nor anti-dominant), can be written as
\begin{equation}\label{eqn:p-adic_expansion}
  (r,s) = (x,y) + p^k(r_0,s_0),
\end{equation}
where $$(x,y) = \sum_{i=0}^{k-1}p^i(a_i,b_i),$$ $r_0 + s_0 = -1$,
and each $(a_i,b_i) \in X_1(T)$. We can see that $r+s=-1$ if and only if
$a_i + b_i = p-1$ for all $i=0,\dots,k-1$. 
Moreover, if $r + s \neq -1$, then we can assume that 
$a_{k-1} + b_{k-1} \neq p-1$. 
\begin{rem}
To see why the last statement is true, first we observe that if $r+s \neq -1$, then there must exist an
integer $1\leq j \leq k$ such that $a_{j-1} + b_{j-1} \neq p-1$ and $a_i + b_i =p-1$ for all $i \geq j$. 
Now set
\[
  (r_1,s_1) = \left( \sum_{i=j}^{k-1}p^{i-j}(a_i,b_i) \right)+ p^{k-j}(r_0,s_0)
\]
and let 
\[
  (x',y') = \sum_{i=0}^{j-1}p^i(a_i,b_i).
\]
Then we can write $(r,s) = (x',y') + p^j(r_1,s_1)$ where $a_{j-1}+b_{j-1}\neq 1$ and $r_1 + s_1 = -1$. 
\end{rem}

Using these $p$-adic expansions, we can state the following lemma which, as we will soon show,
reduces the problem of determining where $S^i(r,s)\neq 0$, $T^i(r,s) \neq 0$ for arbitrary
weights to only having to consider the weights which lie on the fundamental line. 

\begin{lem}\label{lem:fundline_reduction}
  Let $(r,s)$ be a weight which satisfies $(r+1)(s+1) \leq 0$ and $r+s \neq -1$, so that we can write 
  $(r,s) = (x,y) + p^k(r_0,s_0)$ 
  with $(x,y)$ as in \eqref{eqn:p-adic_expansion}  and $a_{k-1} + b_{k-1} = -1$.
   If $x+y >p^k-1$, then
  \begin{align*}
    T^i(r,s) &= T^i(r_0,s_0) \\
    S^i(r,s) &= -T^i(r_0,s_0)
  \end{align*}
  and when $x+y <p^k-1$,
    \begin{align*}
    T^i(r,s) &= -S^i(r_0,s_0) \\
    S^i(r,s) &=  S^i(r_0,s_0).
  \end{align*}
\end{lem}
\begin{proof}
  Notice that when $k=1$, 
  $(r,s) = (a, b) + p(r_0,s_0)$,  in which case the identities follow immediately from 
  Propositions~\ref{prop:bottom_alcove} and \ref{prop:top_alcove}.
  Next we proceed by induction on $k$.  First we consider the case when 
  $x + y < p^k-1$.  This implies that $a_{k-1} + b_{k-1} \leq p-1$.
  However, by assumption we know that $a_{k-1} + b_{k-1}\neq p-1$, and  thus
  \[ a_{k-1}+b_{k-1} < p-1.\] 
  Suppose now that  the 
  above formulas hold whenever $(r,s) = (x,y) + p^k(r_0,s_0)$. Thus, for some arbitrary $(a,b) \in X_1(T)$ we consider 
  $(x',y') = (a,b) + p(x,y)$ and $(r',s') = (x',y') + p^{k+1}(r_0,s_0)$, then we 
  can see that 
  \[
    (r',s') = (a,b) + p(r,s).
  \]
  Since $r+s \neq -1$, then by Lemma~\ref{lem:constant_formula} we get 
  \begin{align*}
    T^i(r',s') &= T^i(r,s) = -S^i(r_0,s_0)\\
    S^i(r',s') &= S^i(r,s) = S^i(r_0,s_0).
  \end{align*}
  Therefore, the identity follows by induction. 
  The case where $x+y > p^k-1$ is proved in a similar way.  
\end{proof}

\subsection{ }
Before we state the main result of this section, we will recall \cite[Theorem 3.6]{and21979}
which states that
$\delta^i(r,s) \neq 0$ for $i=1,2$ with $r\geq s$ if and only if 
$(r,s) = (a,b) + p^n(t, -t-1)$ for some $1\leq t \leq p-1$, $n \geq 1$ and $0\leq a,b \leq p^n-2$. These weights occur inside the 
interiors of the shaded regions depicted in 
Figure~\ref{pic:p=5} when $p=5$ and $n=1,2$. 

\begin{thm}\label{thm:reg_result}
Assume that $p \geq 3$ and let $(r,s) \in X(T)$ be a regular weight such that $\delta^i(r,s)=\dim H^i(r,s) \neq 0$ for
$i=1,2$. Then $D_{\zeta}^i(r,s)=\dim_{\zeta} H^i(r,s) \neq 0$ for $i=1,2$.
\end{thm}
\begin{proof}
%
  Let $(r',s') \in X(T)$ be any weight satisfying $\delta^i(r',s') \neq 0$ for $i=1,2$. 
   By Proposition~\ref{prop:general-char-nonrecursive},
  we may assume that
  $(r',s') = (pr,ps)$ or $(r',s') = (p-2 + pr,p-2+ps)$ where 
  \begin{align*}
    D^i_{\zeta}(pr,ps) &= S^i(r,s)\\ 
    D^i_{\zeta}(p-2+pr, p-2+ps)&=T^i(r,s).
  \end{align*}
 By applying $\tau$ as in Section~\ref{sec:transposition}, we can see that $D^i_{\zeta}(r',s') = D^i_{\zeta}(s',r')$, thus  
 we may assume $r'\geq s'$.
  From the discussion immediately preceding this theorem, it follows that 
  $$(r',s') = (a',b') + p^{n+1}(t, -t-1)$$ for some $1\leq t \leq p-1$, $n \geq 0$ and $0\leq a',b' \leq p^{n+1}-2$. 
 Thus,
   \begin{equation}\label{eqn:stuff}
  (r,s) = (a,b) + p^{n}(t, -t-1)
  \end{equation}
   for some $(a,b) \in X_n(T)$, where we define $X_0(T) := \{(0,0)\}$. 

  So the problem
  reduces to showing that  $S^i(r,s) \neq 0$ and $T^i(r,s) \neq 0$ for $i=1,2$ whenever $(r,s)$ is as in 
  \eqref{eqn:stuff}.   
  Now since $(r+1)(s+1) \leq 0$, then for some $0 \leq k \leq n$, we can write
  \[
  (r,s) = (x,y) + p^k(r_0,s_0)
  \]
  as in \eqref{eqn:p-adic_expansion}.
  It follows that 
  \[
    \{ |S^i(r_0,s_0)|, \, |T^i(r_0,s_0)| \} =\{ |S^i(r,s)|, \, |T^i(r,s)| \} 
  \]
  for $i=1,2$. 
 So without loss of generality, we can assume that $r+s=-1$, or equivalently, 
 \[
 (r,s) = (r,-r-1)
 \]
 for some $r \in \mathbb{Z}$. 
 Furthermore,  since $S^i(r,s) = S^i(s,r)$ and $T^i(r,s) = T^i(s,r)$, then by Serre Duality we get $S^2(r,-r-1)=T^1(r,-r-1)$.
  It suffices to show 
  $S^1(r,-r-1) \neq 0$ and $T^1(r,-r-1) \neq 0$ for any $r \geq 1$. 
 By using Weyl's character formula (see  \cite[Proposition II.5.10]{jan2003})) and the above Serre duality statement,  it can be verified that 
 \[
 T^1(r,-r-1) -S^1(r,-r-1) = 1. 
 \]
 
 If we  assume that either $S^1(r,-r-1) = 0$ or $T^1(r,-r-1) = 0$, then
 \[
   T^1(r,-r-1) =1 + S^1(r,-r-1)
 \]
  implies that
  \[
    S^1(r,-r-1) = 0  \iff T^1(r,-r-1) =1
  \]
  and  
  \[
    T^1(r,-r-1) = 0 \iff S^1(r,-r-1) = -1.
    \]
  It follows that in either case
  we would have 
  \[
    S^1(r,-r-1) + T^1(r,-r-1) \leq 1
  \]
  and hence
  \[
    \phi^1(r,-r-1) + \psi^1(r,-r-1) =S^1(r,-r-1) + T^1(r,-r-1) \leq 1.
  \]
  
  However, by direct computation we can see that when $(r,-r-1) = (t,-t-1)$ with $1 \leq t \leq p-1$,
  \[
    S^1(t,-t-1) + T^1(t,-t-1) = t^2 + t + 1 >1.
  \]
  Thus, for any $(r,-r-1) = (a + p^nt, p^n-1 -a + p^n(-t-1))$ with $0 \leq a \leq p^n-1$, it follows
  inductively from Lemma~\ref{lem: monotonicity} that 
  \[
    S^1(r,-r-1) + T^1(r,-r-1) > 1.
  \]  
  Which gives a contradiction. Therefore, both $S^1(r,-r-1) >0$ and $T^1(r,-r-1) >0$ for all $r\geq 1$. 
\end{proof}

\section{Calculations for subregular weights}\label{sec:sreg_case}
\subsection{ }
In this section, we will show in Proposition~\ref{thm:proj_result} that if $\text{char}(k) =p \geq 3$ is arbitrary and 
$\lambda \in X(T)$ is any weight with $H^i(\lambda)\neq 0$ for some $i\geq 0$ , 
then $H^i(\lambda)|_{G_1}$ is not projective if and only if $\lambda$ does not lie in the Steinberg block. 
As we shall see, this fact, combined with the structure of the $G$-orbit closures on $\mathcal{N}$ and 
Theorem~\ref{thm:reg_result}, will uniquely determine the support varieties $V_{G_1}(H^i(\lambda))$ for any weight.

\subsection{Nilpotent orbits for $G=SL_3(k)$ }
By nilpotent orbit theory, it is well known that for all
primes $p$,
$\mathcal{N}$ has two nonzero orbit closures 
\[
  \mathcal{N} =\overline{\mathcal{O}} \supseteq \overline{\mathcal{O}}_{sreg} \supseteq \{0\}
\]
given by
\[ \mathcal{O} = G \cdot \begin{pmatrix} 
                            0 & 1 & 0 \\
                            0 & 0 & 1 \\
                            0 & 0 & 0 
                          \end{pmatrix}, \quad \mathcal{O}_{sreg} = G \cdot \begin{pmatrix} 
                                                                             0 & 1 & 0 \\
                                                                             0 & 0 & 0 \\
                                                                             0 & 0 & 0
                                                                           \end{pmatrix}.
\]

In Theorem~\ref{thm:reg_result} we demonstrated that when $\lambda$ is regular, 
$V_{G_1}(H^i(\lambda))=\overline{\mathcal{O}}$. Now assume that $\lambda$ is subregular,
then by Proposition~\ref{prop:npv_result}, it can be shown that
either $V_{G_1}(H^i(\lambda)) = \overline{\mathcal{O}}_{sreg}$, or $V_{G_1}(H^i(\lambda)) = \{0\}$
(i.e., $H^i(\lambda)$ is a projective $G_1$-module). 
Therefore, if $H^i(\lambda)|_{G_1}$ is not projective, then
$V_{G_1}(H^i(\lambda))=\overline{\mathcal{O}}_{sreg}$.  
This argument also shows that
if $\lambda$ is in the Steinberg block, then $V_{G_1}(H^i(\lambda))= \{0\}$. 
The computation of the support varieties $V_{G_1}(H^i(\lambda))$ is now reduced to
showing that $H^i(\lambda)|_{G_1}\neq 0$ is not projective provided that $\lambda$
is a subregular, non-Steinberg weight. 

 By applying \cite[Theorem 3.4.1]{npv2002}, we can see that if 
 $\psi_p(t)^s \nmid \dim_t H^i(\lambda)$, then 
 $$\dim V_{G_1} (H^i(\lambda)) \geq 6 -2(s-1).$$ 
 Thus, if $s \leq 3$, then  $\dim V_{G_1}(H^i(\lambda)) > 0$ since $p\geq 3$ implies that 
 $d(\Phi,p) = 0$.
 So to prove that $H^i(\lambda)$ is not projective as a $G_1$-module, 
 we only need to show that $\psi_p(t)^3 \nmid \dim_t H^i(\lambda)$. 
 
 Now if $\mathcal{L}(\lambda)$ has exactly one non-vanishing cohomology group, then 
 $\dim_t H^i(\lambda)$ is given by Weyl's generic dimension formula which can be used to show that 
 $s=2$. 
 To handle the case in which $\mathcal{L}(\lambda)$ has multiple non-vanishing cohomology groups,
 we will use the character formulas given in Section~\ref{sec:dim_formulas}.

 \subsection{ }
We observe that if $\lambda \in X(T)$ is any subregular, non-Steinberg weight, then 
 for some $(r,s) \in X(T)$ and $0\leq a \leq p-2$,  
  $\lambda$ must be one of the following:
\begin{align*}
  (p-1+pr, a +ps) \\
  (p-2-a + pr, p-1 + ps)\\
  (a+pr, p-2-a + ps)
\end{align*}
The character formulas in Proposition~\ref{prop:sreg_formulas} can be specialized to 
give the following generic dimension formulas. 
\begin{prop}
  Let $p\geq 2$ and $(r,s) \in X(T)$ be arbitrary, then
\begin{align*}
 D^i_t(p-1+pr, a+ps) &= D_t(p-1,a)D^i_{t^p}(r,s) + D_t(p-2-a,p-1)D_{t^p}^i(r,s-1) \\
          &\, + D_t(a,p-2-a)D^i_{\alpha,t^p}(r+1,s-1) \\
 D^i_t(p-2-a+pr, p-1+ps) &= D_t(p-2-a,p-1)D^i_{t^p}(r,s) + D_t(p-1,a)D_{t^p}^i(r-1,s) \\
          &\, + D_t(a,p-2-a)D^i_{\beta,t^p}(r-1,s+1) \\
 D^i_t(a+pr, p-2-a+ps) &= D_t(a,p-2-a)D^i_{t^p}(r,s) + D_t(p-1,a)D_{t^p}^i(r,s-1) \\
                           &\, + D_t(p-2-a,p-1)D_{t^p}(r-1,s) \\
                           &\,+  D_t(a,p-2-a)D^i_{t^p}(r-1,s-1)
\end{align*}
where $0 \leq a \leq p-2$.
\end{prop}

For notational convenience, we shall define 
\begin{align*}
  h_{1,i}(t)&=D^i_t(p-1+pr, a+ps),\\ 
  h_{2,i}(t)&=D^i_t(p-2-a+pr, p-1+ps),\\ 
  h_{3,i}(t)&=D^i_t(a+pr, p-2-a+ps). 
\end{align*} 
So if $H^i(\lambda)$ is 
projective, then we must have that $\phi_p(t)^3 \mid h_j(t)$ and hence, 
$h_{j,i}'(\zeta) = h_{j,i}''(\zeta) = 0$ for some $j$. 
By explicitly calculating $h'(\zeta)$ and  $h_j''(\zeta)$, we will show that this only occurs when $H^i(\lambda)=0$.

We begin by calculating $h_j'(\zeta)$:
\begin{align*}
  h_{1,i}'(\zeta) &= \frac{2p(\zeta^{a+1}-\zeta^{-a-1})}{\zeta(\zeta-\zeta^{-1})^2(\zeta^2-\zeta^{-2})}
                    \left(\delta^i(r,s) + \delta^i(r,s-1)-\delta^i_{\alpha}(r+1,s-1)\right), \\
  h_{2,i}'(\zeta) &= \frac{2p(\zeta^{a+1}-\zeta^{-a-1})}{\zeta(\zeta-\zeta^{-1})^2(\zeta^2-\zeta^{-2})}
                    \left(\delta^i(r,s) + \delta^i(r-1,s)-\delta^i_{\beta}(r-1,s+1)\right), \\
  h_{3,i}'(\zeta) &= \frac{2p(\zeta^{a+1}-\zeta^{-a-1})}{\zeta(\zeta-\zeta^{-1})^2(\zeta^2-\zeta^{-2})}
                    \left(-\delta^i(r,s) -\delta^i(r-1,s-1)+ \delta^i(r,s-1)+\delta^i(r-1,s)\right).
\end{align*}
For simplicity, we define the integers
\begin{align*}
       Q_1^i(r,s) &= \delta^i(r,s) + \delta^i(r,s-1)-\delta^i_{\alpha}(r+1,s-1), \\
       Q_2^i(r,s) &= \delta^i(r,s) + \delta^i(r-1,s)-\delta^i_{\beta}(r-1,s+1), \\
       Q_3^i(r,s) &= -\delta^i(r,s) -\delta^i(r-1,s-1)+ \delta^i(r,s-1)+\delta^i(r-1,s).
\end{align*}
We observe that $h_{j,i}'(\zeta) = 0$ if and only if $Q_j^i(r,s) = 0$.


We now employ Lemma~\ref{lem:deriv-vanish} to calculate $h_{j,i}''(\zeta)$. 
First set
\begin{align*}
  g(t) &= (t-t^{-1})^2(t^2-t^{-2})\\
  p(t) &= (t^p - t^{-p}),\\
  f_1(t) &= (t^{a+1}-t^{-a-1})(t^{p+a+1}-t^{-p-a-1}) \\
  f_2(t) &= (t^ {p-1-a}-t^{-p+1+a})(t^{2p-1-a}-t^{-2p+1+a})\\
  f_3(t) &= (t^{a+1}-t^{-a-1})(t^{p-1-a}-t^{-p+1+a})
\end{align*}  
and also set 
\begin{align*}
q_1(t) &= D_t(p-1,a),\\ 
q_2(t) &= D_t(p-2-a, p-1),\\
q_3(t) &= D_t(a, p-2-a)
\end{align*}
so that
\[
  q_j(t) = \frac{p(t)f_j(t)}{g(t)}.
\]
Furthermore, since $p(\zeta)=0$ and $f_1(\zeta) = f_2(\zeta)=-f_3(\zeta)$, we get 
\[
  q_j''(\zeta) = \frac{p''(\zeta)f_j(\zeta)g(\zeta) -2p'(\zeta)f_j(\zeta)g'(\zeta) +2p'(\zeta)f_j'(\zeta)g(\zeta)}{g(\zeta)^2}.
  \]
  
 Now the derivatives $f_j'(\zeta)$ are given by
 \begin{align*}
f_1'(\zeta) &= (2a+2+p)(\zeta^{a+1}-\zeta^{-a-1})(\zeta^a + \zeta^{-a-2})\\
f_2'(\zeta) &= (2a+2-3p)(\zeta^{a+1}-\zeta^{-a-1})(\zeta^a + \zeta^{-a-2})\\
f_3'(\zeta) &= -(2a+2-p)(\zeta^{a+1}-\zeta^{-a-1})(\zeta^a + \zeta^{-a-2}).
\end{align*}
  Thus, if we set 
\begin{align*}
u_j(\zeta) &= f_j(\zeta)\frac{p''(\zeta)g(\zeta)-2p'(\zeta)g'(\zeta)}{g(\zeta)^2} \\
v(\zeta)  &= \frac{2p'(\zeta)(\zeta^{a+1}-\zeta^{-a-1})(\zeta^a + \zeta^{-a-2})}{g(\zeta)},
\end{align*}
then  $u(\zeta) = u_1(\zeta)=u_2(\zeta)=-u_3(\zeta)$ and
\begin{align*}
q_1''(\zeta) & = u(\zeta) + (2a+2+p)v(\zeta) \\
q_2''(\zeta) & = u(\zeta) + (2a+2-3p)v(\zeta) \\
q_3''(\zeta) &=  -u(\zeta) - (2a+2 -p)v(\zeta).
\end{align*}

Using the fact that $q_j(\zeta) =0$ and $(\dim_t M^{(1)})'(\zeta) = 0$ for any $G$-module $M$, 
we get
\begin{align*}
  h_{1,i}''(\zeta) & = u(\zeta)Q_1(r,s) + v(\zeta)( (2a+ 2 + p)\delta^i(r,s) +(2a+2-3p)\delta^i(r,s-1) \\
                               &- (2a+2-p)\delta^i_{\alpha}(r+1,s-1) ) \\
  h_{2,i}''(\zeta) & = u(\zeta)Q_2(r,s) + v(\zeta)( (2a+ 2 - 3p)\delta^i(r,s) +(2a+2+p)\delta^i(r-1,s)  \\
                          &- (2a+2-p)\delta^i_{\beta}(r-1,s+1) ) \\
  h_{3,i}''(\zeta) & = -u(\zeta)Q_3(r,s) + v(\zeta)( -(2a+ 2 - p)\delta^i(r,s) -(2a+2-p)\delta^i(r-1,s-1)  \\ 
                           &+(2a+2+p)\delta^i(r,s-1) + (2a+2-3p)\delta^i(r-1,s) ).
 \end{align*}
We define the integers
\begin{align*}
       R_1^i(r,s) &= \delta^i(r,s) - \delta^i(r,s-1) \\
       R_2^i(r,s) &= \delta^i(r,s) - \delta^i(r-1,s) \\
       R_3^i(r,s) &= \delta^i(r,s-1) -\delta^i(r-1,s).
\end{align*}

It can be immediately verified that $\psi_p(t)^3 \nmid h_{j,i}(t)$ if and only if
$Q_j^i(r,s)$ and $R_j^i(r,s)$ are not simultaneously zero.
 In the remaining part of this section, we will verify this for each 
$j$.

\subsection{ }\label{subsection:projective-equiv}
  We begin by observing that, unlike the case  when $\lambda$ is a regular weight,
  the multiplicity  of $\psi_p(t)$ as a divisor of $\dim_t(H^i(\lambda))$ is not constant
  across all $i$ such that $H^i(\lambda) \neq 0$.
  For instance, if we let $p=7$, then
  $Q_3^1(27,-22) =0$, even though 
  \[
   H^1(a + 7(27),p-2-a + 7(-22))\neq 0
  \]
  for any $a$ with  $0 \leq a \leq 5$, so the multiplicity is at least $2$ in this case. However,  $Q_3^2(27,-22)\neq 0$ 
  so that when $i=2$, the multiplicity is 1.

\subsection{Calculations for weights of the form $(p-1+pr,a-1 +ps)$}\label{sec:Q_1} 
Suppose that 
$$\lambda = (p-1+pr, a-1 + ps)$$ for some $0\leq a \leq p-2$ and 
$H^i(p-1+pr,a-1 + ps) \neq 0$, then to show that $H^i(p-1+pr,a-1+ps)|_{G_1}$ 
is not projective, it 
will be enough to show the following two expressions cannot vanish simultaneously
\begin{align*}
  Q_1^i(r,s) &= \delta^i(r,s)+\delta^i(r,s-1) -\delta^i_{\alpha}(r+1,s-1)\\
  R_1^i(r,s) &= \delta^i(r,s) - \delta^i(r,s-1).
\end{align*}

However, it can be verified by using essentially the same techniques that
were used in Section~\ref{sec:reg_case} to show $S^i(r,s) \neq 0$,
that 
$\delta^i(r,s)-\delta^i(r,s-1) \neq 0$ whenever 
$\delta^i(r,s) \neq 0$. 
In fact,
\[
  |R_1^i(r_0,s_0)| \leq |R_1^i(r,s)|
\]
for any $(r,s) = (a+pr_0,b+ps_0)$ where $(r_0+1)(s_0+1) <0$ and $(a,b) \in X_1(T)$.
It follows from the result in Section~\ref{subsection:projective-equiv} that
$H^i(p-1+pr,a-1+ps)|_{G_1}$ is not projective if it is nonzero.

\subsection{Calculations for weights of the form $(p-2-a+pr,p-1 +ps)$}\label{sec:Q_2}
Now suppose that $$\lambda = (p-2-a+pr, p-1 + ps)$$ for some $0\leq a \leq p-2$ and 
$H^i(p-2-a+pr,p-1 + ps) \neq 0$, then to show that $H^i(p-2-a+pr,p-1+ps)|_{G_1}$ 
is not projective, it is sufficient 
to show the following two expressions cannot  simultaneously equal zero
\begin{align*}
  Q_2^i(r,s) &= \delta^i(r,s)+\delta^i(r-1,s) -\delta^i_{\beta}(r-1,s+1)\\
  R_2^i(r,s) &= \delta^i(r,s) - \delta^i(r-1,s).
\end{align*}

By applying $\tau$ as in Section~\ref{sec:transposition}, we see that 
\[
  R_2^i(r,s) = R_1^i(s,r) \neq 0.
\]
So again, by the result in Section~\ref{subsection:projective-equiv},
we get that if $H^i(p-2-a+pr,p-1 +ps)|_{G_1}$ is not projective then it is nonzero.

\subsection{Calculations for weights of the form $(a+pr,p-2-a +ps)$}\label{sec:Q_3}
Finally, we consider the case when 
$$\lambda = (a + pr, p-2-a + ps)$$ for some $0\leq a \leq p-2$ and 
$H^i(a+pr,p-2-a + ps) \neq 0$.  To show, as before,  that $H^i(a+pr,p-2-a+ps)|_{G_1}$ 
is not projective, it 
will be enough to check that the following two expressions cannot simultaneously vanish
\begin{align*}
  Q^i_3(r,s) &= -\delta^i(r,s)-\delta^i(r-1,s-1) + \delta^i(r,s-1)+\delta^i(r-1,s)\\
  R^i_3(r,s) &= \delta^i(r,s-1)-\delta^i(r-1,s).
\end{align*}
However, unlike the previous cases, we cannot always ensure that $R_3^i(r,s) \neq 0$ if $\delta^i(r,s) \neq 0$.
Instead we will demonstrate that whenever $Q_3^i(r,s) =0$, then it must follow that $R_3^i(r,s) \neq 0$.  

We proceed by first determining the precise weights $(r,s)$ where $Q_3^i(r,s)=0$. This is done in 
Proposition~\ref{prop:Q_3_vanish}, the proof of this proposition occupies the first four subsections of 
Section~\ref{sec:verif}. The main idea is to develop recursive expansion formulas for $Q_3^i(r,s)$
(cf. Proposition~\ref{prop:Q_3_formula}). In Proposition~\ref{prop:Q_3_fundline} we use these formulas to show 
that $Q_3^i(x,-x-1) \neq 0$ for $x \in \mathbb{Z}$ with $x \neq 0$. To extend this result to more types of weights, 
we generalize the expansion formula for $Q_3^i(r,s)$ in Lemma~\ref{prop:general_formula}. 
Then, with some additional technical lemmas, we are able to prove  Proposition~\ref{prop:Q_3_vanish}.

 Finally, in Proposition~\ref{prop:R_3_vanish} we show that $R_3^i(r,s)\neq 0$ whenever $Q_3^i(r,s)=0$.
 The proof follows from direct computation, which is made possible by the explicit description of the weights.   

\subsection{ }
We have just shown that in all three cases for $j$, $Q^i_j(r,s)$ and $R^i_j(r,s)$ do not simultaneously 
vanish. The next proposition is the main result of this section, it gives precise conditions for when $H^i(r,s)$ 
is a projective $G_1$-module. 
\begin{prop}\label{thm:proj_result}
  If $H^i(r,s) \neq 0$, then it is projective as a $G_1$-module  if and only if  $(r,s) \in X(T)$ is in the Steinberg block. 
\end{prop}
\begin{proof}
If $(r,s)$ is in the Steinberg block, then by Proposition~\ref{prop:npv_result} $V_{G_1}(H^i(r,s))=0$ and
 thus $H^i(r,s)|_{G_1}$ is projective. 
 Now for the other direction, we first note that 
  the case where $H^i(r,s) \neq 0$ for only one $i$ also follows from Proposition~\ref{prop:npv_result}.
  We can reduce to the case when $H^i(r,s)\neq 0$ for multiple $i$.   
  If $(r,s)$ is regular, the result is covered in Theorem~\ref{thm:reg_result}. The subregular case now follows
  by making the observation that if we write $(r,s) = (a+pr_0,b+ps_0)$ with $(a,b) \in X_1(T)$, then 
  it was shown in  this section that $R^i_j(r_0,s_0)$ and $S^i_j(r_0,s_0)$ cannot be simultaneously zero for $j=1,2,3$
  and thus $\psi_p(t)^3 \nmid \dim_t H^i(r,s)$.
\end{proof}

\section{The $p=2$ case}\label{sec:alt}
\subsection{ }
There are a large class of weights with multiple non-vanishing cohomology groups whose support varieties can be computed 
by using somewhat more elementary techniques. In the case when $p=2$, this method will determine the support varieties for every weight. 
\subsection{An alternative method }
We shall assume throughout that $\text{char}(k)=p\geq 2$.
Let $E=L(0,1)$ be the dual of the standard representation, with weights
$(0,1)$, $(1,-1)$ and $(-1,0)$.  Now let $\lambda =(x,y) \in X(T)$ be any weight such that 
\[
\lambda = \mu + p^n(1,-1).
\]
Then $\mu=(x-p^n,y+p^n)$ has the property that 
\begin{align*}
\mu+(0,p^n)     &= (x-p^n, y + 2p^n) = \lambda + p^n\beta \in X(T)_+ \\
\mu + (-p^n,0) &= (x-2p^n,y+p^n)  = \lambda -p^n\alpha \in -X(T)_+.
\end{align*}

For example, if take any $(a,b) \in X_n(T)$, then the weight
\[
\lambda = (a,b) -p^n\beta
\]
satisfies all of these properties, since 
\begin{align*}
\lambda + p^n\beta &= (a,b) \in X(T)_+ \\
\lambda -p^n\alpha &= (a-p^n, b-p^n) \in -X(T)_+.
\end{align*} 

Now  we build two short exact sequences of $B$-modules
\begin{gather*}
 0 \rightarrow V^{(n)} \rightarrow L(0,p^n) \xrightarrow{\pi} (0,p^n) \rightarrow 0\\
 0 \rightarrow (-p^n,0) \rightarrow V^{(n)} \rightarrow (p^n,-p^n) \rightarrow 0
 \end{gather*}
 where $\pi: L(0,p^n) = L(0,1)^{(n)} \rightarrow (0,p^n)$ is the quotient map. 
If we tensor these short exact sequences with $\mu$, we get 
\begin{gather*}
 0 \rightarrow V^{(n)}\otimes \mu \rightarrow L(0,p^n)\otimes \mu \xrightarrow{\pi} \lambda + p^n\beta \rightarrow 0\\
 0 \rightarrow \lambda-p^n\alpha \rightarrow V^{(n)}\otimes \mu \rightarrow \lambda \rightarrow 0.
 \end{gather*}

Now apply the induction functor $\text{ind}_B^G-$  to the first short exact sequence to get 
\[
  0 \rightarrow H^0(\lambda + p^n\beta) \rightarrow H^1(V^{(n)}\otimes \mu) \rightarrow L(0,p^n)\otimes H^1(\mu) \rightarrow 0,
\]
where we used the fact that $H^0(\mu)=0$ since $\mu$ is non-dominant.
We  can also  see that the second short exact sequence gives us an isomorphism
\[
0 \rightarrow H^1(V^{(n)}\otimes\mu) \xrightarrow{~} H^1(\lambda) \rightarrow 0
\]
since $\lambda-p^n\alpha$ is anti-dominant. Using the isomorphism and substituting yields
\[
0\rightarrow H^0(\lambda + p^n\beta) \rightarrow H^1(\lambda) \rightarrow L(0,p^n)\otimes H^1(\mu) \rightarrow 0.
\]
So we observe that on the level on the characters
\[
\chi^1(\lambda) = \chi_p(0,p^n)\chi^1(\mu) + \chi^0(\lambda+p^n\beta).
\]

Now suppose that $p\geq 3$ and let  $\lambda \in X(T)$ be a regular weight., then by evaluating the generic dimension at a primitive $p^{th}$ root of unity, we get 
\[
D_{\zeta}^1(\lambda) =  3D_{\zeta}^1(\mu) + D_{\zeta}^0(\lambda + p^n\beta).
\]

We may assume without loss of generality that $\lambda$ is in the 0-block, so that the quantum quantum dimensions above are integers. 
Using Weyl's generic dimension formula, one can verify that for any dominant  $\nu \in X(T)_+$ in the $0$-block 
$|D_{\zeta}^0(\nu)| =1$.
Thus
\[
D^1_{\zeta}(\lambda) = \pm 1 \text{ mod } 3
\]
and in particular, $D^1_{\zeta}(\lambda) \neq 0$ which implies $V_{G_1}(H^1(\lambda))=V_{G_1}$. 

Suppose now that $p \geq 2$ and that $\lambda$ is a subregular, non-Steinberg weight, then by Lemma~\ref{lem:deriv-vanish},
 setting
$f(t) = \dim_t L(0,p^n)$ gives us
\begin{align*}
   (f(t)D^1_t(\mu))'(\zeta) &= f'(\zeta)D^1_{\zeta}(\mu) + f(\zeta)(D^1_t(\mu))'(\zeta) \\
                            &= 3(D^1_t(\mu))'(\zeta).
\end{align*}
Thus
\[
  (D^1_t(\lambda))'(\zeta) = 3(D^1_t(\mu))'(\zeta) + (D^0_t(\lambda + p^n\beta))'(\zeta). 
\]
And by using a $\text{mod } 3$ argument we also get $(D^1_t(\lambda))'(\zeta) \neq 0$. Therefore 
$V_{G_1}(H^1(\lambda)) = \overline{\mathcal{O}_{sreg}}$. 

From Serre duality and by applying the automorphism $\tau$ from Section~\ref{sec:transposition},
we see that for any weight of the form 
\begin{equation}\label{eqn:p=2form}
  \lambda = (a,b) - p^n\beta = (a,b) - p^n(1,-2),
\end{equation}
$V_{G_1}(H^i(\lambda)) = V_{G_1}(H^0(w\cdot \lambda))$ where $w\in W$ and $w\cdot \lambda \in X(T)_+$. 

\subsection{A result for $p=2$.}
When $p=2$, we know by~\cite[Theorem 3.6]{and21979} that all of the weights $\lambda$ with two 
non-vanishing cohomology groups are given by Equation~\ref{eqn:p=2form}. Therefore, Theorem~\ref{thm:main_result} is immediately
verified for $p=2$. 


\section{Conjectures and open problems}\label{sec:conj}
\subsection{ }
For $p$-good the support varieties $V_{G_1}(H^0(\lambda))$ were precisely determined for all $\lambda \in X(T)_+$ 
 in \cite[Theorem 6.2.1]{npv2002} and for $p$-bad as well (cf.  \cite[Theorem 3.6]{uga2007} for the classical types and
 the tables in \cite[Section 4.2]{uga2007} for the exceptional cases).  
 The following conjecture extends the \cite{npv2002} result to the higher cohomology groups.
It may be thought of as an analogue to the Borel-Bott-Weil Theorem for support varieties. 
\begin{conj}
  Let $G$ be semisimple, simply connected with $p$ good, and let $\lambda \in X(T)$ be arbitrary.
  Suppose $w \in W$ is such that $w\cdot \lambda \in X(T)_+$ and that $H^i(\lambda) \neq 0$ for some $i$,
  then
  $V_{G_1}(H^i(\lambda)) = V_{G_1}(H^0(w\cdot \lambda))$.
\end{conj}
The case where $G=SL_2(k)$ actually follows from Proposition~\ref{prop:npv_result} since there are no weights $\lambda$
such that $H^i(\lambda) \neq 0$ and $H^j(\lambda) \neq 0$ with $i\neq j$.
The $G=SL_3(k)$ case is given by Theorem~\ref{thm:main_result}, which is the main result of this paper. 

For general semisimple simply connected groups $G$, we showed in Proposition~\ref{prop:npv_result} 
that for all 
$G$,  $\cup_{i\geq 0} V_{G_1}(H^i(\lambda))= V_{G_1}(H^0(w\cdot\lambda))$,
which implies that 
\[
  V_{G_1}(H^j(\lambda))= V_{G_1}(H^0(w\cdot\lambda))
\] for some $j$. However, the proof does not guarantee equality
for all $i$ where $H^i(\lambda)\neq0$. 
\subsection{ }
To try to handle this general case, one might hope to obtain recursive character formulas of the kind given in this paper. 
A recent preprint (cf. \cite{DONK2012}) works out the formulas in type $B_2$ when $\text{char}(k)=2$.
 Also there have been some partial calculations  for these formulas when 
$G$ is the simple, simply connected group of type $G_2$ and $\text{char}(k) =2$ (cf. \cite{ANWAR}). In that case 
the main obstruction was that the expressions for 
$\text{ch}\,H^i(\lambda)$ involved certain rank 4 vector bundles whose recursive formulas could not be established. 
So it appears that determining explicit formulas for $\text{ch}\, H^i(\lambda)$ when  $G$ is arbitrary is an extremely difficult problem.

\section{Appendix: Verification of formulas for the subregular case}\label{sec:verif}
In this section we will carefully verify a number of the 
technical lemmas which were used in the key results from Section~\ref{sec:sreg_case}.
It is useful to note that the terms $S^i(r,s)$, $T^i(r,s)$, $\phi^i(r,s)$ and $\psi^i(r,s)$ appearing in
this section are the same as in Section~\ref{sec:reg_case}.
\subsection{Recursive expansions and vanishing results for $Q_3^i(r,s)$}
To determine where $Q^i_3(r,s)$ vanishes, we will begin by computing their recursive expansion formulas.
\begin{prop}\label{prop:Q_3_formula}
  Let $(r,s) = (a+pr_0, b + ps_0)$ and $a+b < p-1$. Then
  \[
    Q_3^i(r,s) = -(a+b-(p-1))S^i(r_0,s_0) + pQ_3^i(r_0,s_0).
  \]
  If $a+b =p-1$ then
  \[
    Q_3^i(r,s) = -\frac{1}{2}a(a+1)[S^i(r_0,s_0)+T^i(r_0,s_0)] - \frac{1}{2}p(p-1-2a)\phi^i(r_0,s_0) + pQ_3^i(r_0,s_0).
  \]
  If $a+b > p-1$ then
  \[
    Q_3^i(r,s) = (a+b-(p-1))T^i(r_0,s_0)+pQ_3^i(r_0,s_0).
  \]
  Lastly, if $\phi^i(r_0,s_0)=\psi^i(r_0,s_0)=0$, which occurs whenever $r_0+s_0 \neq -1$, then 
  for all $(a,b)$ 
  \begin{align*}
    Q_3^i(r,s) &= (a+b-(p-1))T^i(r_0,s_0)+pQ_3^i(r_0,s_0) \\
               &= -(a+b-(p-1))S^i(r_0,s_0)+pQ_3^i(r_0,s_0).
  \end{align*}
\end{prop}

\subsection{}
These formulas can be used to show that $Q_3^i(r,s) \neq 0$ for any $(r,s) = (x,-x-1)$ with $x \in \mathbb{Z}$.
We begin with the following lemma. 
\begin{lem}
  If $\phi^i(r,s) \geq 0$ and $\psi^i(r,s) \geq 0$, then 
  \[
    \frac{1}{2}a(a+1)[S^i(r,s)+T^i(r,s)] + \frac{1}{2}p(p-1-2a)\phi^i(r,s)\geq 0. 
  \]
 \end{lem}
 \begin{proof} 
   We first observe that 
   \[
     S^i(r,s) + T^i(r,s) = \phi^i(r,s) + \psi^i(r,s)
   \]
   and so
   \begin{align*}
    &\frac{1}{2}a(a+1)[S^i(r,s)+T^i(r,s)] + \frac{1}{2}p(p-1-2a)\phi^i(r,s)  \\
    &\qquad =\frac{1}{2}[a(a+1) + p(p-1-2a)]\phi^i(r,s) + \frac{1}{2}a(a+1)\psi^i(r,s)\\ 
    &\qquad =\frac{1}{2}[a^2 + (1-2p)a + p^2-p]\phi^i(r,s) + \frac{1}{2}a(a+1)\psi^i(r,s) \\ 
    &\qquad \geq \frac{1}{2}[a^2 + (1-2p)a + p^2-p]\phi^i(r,s). 
  \end{align*}
  Since we are assuming that $\phi^i(r,s) \geq 0$, it will be sufficient to show that 
  \[
   f(a,p)= a^2 + (1-2p)a + p^2-p \geq 0
  \]
  for any $a,p \in \mathbb{Z}$. 
  The idea is that for each fixed prime $p$, we can regard $f(a,p)$ as a quadratic polynomial in $a$. 
  If we now allow $a$ to be any real number, then we get 
  \[
    f(a,p)=\left(a-\frac{2p-1}{2}\right)^2 -\frac{1}{4}.
  \]
  Hence, $f(a,p) \geq -\frac{1}{4}$ for any $a \in \mathbb{R}$. However, since we know that $f(a,p) \in \mathbb{Z}$
  for $p$ prime and $a \in \mathbb{Z}$, it follows that $f(a,p) \geq 0$. 
\end{proof}
We also need the following lemma. 
\begin{lem}\label{lem:some_lemma}
  Suppose $(r,s) = (a + pr_0, p-1-a + ps_0)$ where  $0 \leq a \leq p-1$, $(r_0+1)(s_0+1) < 0$ and
  $Q^i_3(r_0,s_0) < 0$, then 
  \[ 
    Q_3^i(r,s) < Q_3^i(r_0,s_0). 
  \]
  If we replace $<$ with $\leq$ above, this statement also holds. 
\end{lem}
\begin{proof}
Since $Q^i(r,s)$ is given by the formula 
  \[
    Q_3^i(r,s) = -\frac{1}{2}a(a+1)[S^i(r_0,s_0)+T^i(r_0,s_0)] - \frac{1}{2}p(p-1-2a)\phi^i(r_0,s_0) + pQ_3^i(r_0,s_0),
  \]
by the preceding lemma, we know that
\[
-\frac{1}{2}a(a+1)[S^i(r_0,s_0)+T^i(r_0,s_0)] - \frac{1}{2}p(p-1-2a)\phi^i(r_0,s_0) \leq 0.
\]
Therefore, 
\[
  pQ_3^i(r_0,s_0) < Q_3^i(r_0,s_0) < 0.
\]
\end{proof}
These two lemmas can be applied to prove our first non-vanishing result for $Q^i_3(r,s)$. 
\begin{prop}\label{prop:Q_3_fundline}
For any integer $x \neq 0$,
\[
  Q^1_3(x,-x-1) = Q^2_3(x,-x-1) < 0.
\]
\end{prop}
\begin{proof}
  Without loss of generality, we can assume that $x >0$. 
  This follows by applying
  the automorphism $\tau$ which was  introduced in Section~\ref{sec:transposition}.
 We first consider the base case where 
 $1 \leq x \leq p-1$, 
 \[
   Q_3^1(x,-x-1) =-\frac{1}{2}x(x+1) <0.
 \]
 If $x \geq p$, we can write 
 \[
   (x,-x-1) = (a, p-1-a) + p(x_0,-x_0-1)
 \]
 for some $0 \leq a \leq p-1$. Hence, if we assume that 
 $Q_3^1(x_0,-x_0-1) <0$, then by  Lemma~\ref{lem:some_lemma} 
 \[
   Q_3^1(x,-x-1) \leq  Q_3^1(x_0,-x_0-1) < 0.
 \]
 The proposition follows from induction.
 \end{proof}

\subsection{ }
 It is useful to observe that there
 are some interesting similarities between the fundamental line expansions
  for $S^i(r,s)$, $T^i(r,s)$ and $Q_3^i(r,s)$. Notice that if we define 
 \[
  \theta^i(r_0,s_0,a) = \frac{1}{2}(\phi^i(r_0,s_0) + \psi^i(r_0,s_0)) + \frac{1}{2}p(p-1-2a)\phi^i(r_0,s_0),
\]
 then for $(r,s) = (a + pr_0, p-1-a+ps_0)$,
 \begin{equation}\label{eq:fund_lineQ_3}
   \begin{aligned}
   S^i(r,s) &= \theta^i(r_0,s_0,a) + S^i(r_0,s_0) \\
   T^i(r,s) &= \theta^i(r_0,s_0,a) + T^i(r_0,s_0) \\
   Q_3^i(r,s) &= -\theta^i(r_0,s_0,a) + pQ_3^i(r_0,s_0). \\
 \end{aligned}
 \end{equation}

By starting from the base case, which is given by  
\begin{equation}\label{eq:base_caseQ_3}
\begin{aligned}
  S^1(x,-x-1) &=\frac{1}{2}x(x+1) \\
  T^1(x,-x-1) &=\frac{1}{2}x(x+1)+1 \\
  Q_3^1(x,-x-1) &=-\frac{1}{2}x(x+1)
\end{aligned}
\end{equation}
for $1 \leq x \leq p-1$,
we can see that as $r >0$ increases,  $Q^i_3(r,s)$  grows much faster in magnitude than the $T^i(r,s)$ and $S^i(r,s)$ terms.

We even get the following lemma, which is analogous to Lemma~\ref{lem:fundline_reduction}.
This will allow us to extend our vanishing results to certain weights which do not lie on
the fundamental line. 
\begin{lem}\label{prop:general_formula}
  Suppose $(r,s)$ satisfies $(r+1)(s+1) < 0$ and $r+s \neq -1$, so that we can write  $(r,s) = (x,y) + p^k(r_0,s_0)$ as in \eqref{eqn:p-adic_expansion}, 
  where $$(x,y) = \sum_{i=0}^{k-1}p^i(a_i,b_i)$$ 
   and   $a_{k-1} + b_{k-1} \neq p-1$.
  Then if $x+y <p^k-1$,
  \begin{align*}
    Q_3^i(r,s) &= -(x+y-(p^k-1))S^i(r_0,s_0) + p^kQ^i_3(r_0,s_0). 
  \end{align*}
  and if $x+y  > p^k-1$,
  \begin{align*}
    Q_3^i(r,s) &= (x+y-(p^k-1))T^i(r_0,s_0) + p^kQ^i_3(r_0,s_0)
  \end{align*}
\end{lem}
\begin{proof}
  The argument is identical to the proof of Lemma~\ref{lem:fundline_reduction}, but instead
  we use the identity 
  \begin{align*}
    Q^i_3(r',s') &=-(a+b-(p-1))S^i(r,s) + pQ_3^i(r,s)\\
    &= -(a+b-(p-1))S^i(r_0,s_0) + p[-((x+y)-(p^k-1))S^i(r_0,s_0) + p^kQ^i_3(r_0,s_0)] \\
    &= -(a+b-(p-1) + px+py +p^{k+1}-p)S^i(r_0,s_0) +  p^{k+1}Q^i_3(r_0,s_0) \\
    & =-((x'+y')-(p^{k+1}-1))S^i(r_0,s_0) + p^{k+1}Q_3^i(r_0,s_0)
  \end{align*}
  which was given in Proposition~\ref{prop:Q_3_formula}.
\end{proof}

\subsection{ }
We now want to use this lemma to get additional non-vanishing results for 
$Q^i_3(r,s)$. 
\begin{cor}\label{cor:reduce_to_base}
  Let $(r,s)=(x,y) + p^k(r_0,s_0)$ with $a_{k-1}+b_{k-1} \neq p-1$ and $r_0 + s_0 = -1$ be as
  in \eqref{eqn:p-adic_expansion}. Then
  \begin{enumerate}
    \item[a.] if $x+y > p^k-1$ and $T^i(r_0,s_0) \leq -Q_3^i(r_0,s_0)$, then $Q_3^i(r,s) <0$
    \item[b.] if $x+y < p^k-1$ and $S^i(r_0,s_0) \leq -Q_3^i(r_0,s_0)$, then $Q_3^i(r,s) < 0$. 
  \end{enumerate}
\end{cor}
\begin{proof}
  First suppose $a_{k-1}+b_{k-1} > p-1$. By Lemma~\ref{prop:general_formula}, we get
  \[
    Q^i_3(r,s) = ((x+y)-(p^k-1))T^i(r_0,s_0)  + p^kQ_3^i(r_0,s_0).
  \]
  Therefore, $Q^i_3(r,s) <0$ if and only if 
  \[
    ((x+y)-(p^k-1))T^i(r_0,s_0) < -p^kQ_3^{i}(r_0,s_0).
  \]
  Since $(x,y) \in X_k(T)$ and  $x + y > p^k-1$, 
  we know that $0<x+y -(p^k-1) \leq (p^k-1)$ which allows us to
  conclude that 
  \[
    ((x+y)-(p^k-1))T^i(r_0,s_0) \leq (p^k-1)S^i(r_0,s_0)<p^kT^i(r_0,s_0).
  \]
  Thus if $T^i(r_0,s_0) \leq Q_3^i(r_0,s_0)$, then $Q_3^i(r,s) < 0$. 

  Now consider the case where $a_{k-1}+b_{k-1} < p-1$,
  \[
    Q^i_3(r,s) = -((x+y)-(p^k-1))S^i(r_0,s_0)  + p^kQ_3^i(r_0,s_0).
  \]
  This means $Q^i_3(r,s) <0$ if and only if 
  \[
    -((x+y)-(p^k-1))S^i(r_0,s_0) < -p^kQ_3^{i}(r_0,s_0).
  \]
  Again since $(x,y) \in X_k(T)$ and $x + y < p^k-1$, 
  we get $-(p^k-1) < x+y -(p^k-1) < 0$ and hence
  \[
    -((x+y)-(p^k-1))S^i(r_0,s_0) \leq (p^k-1)S^i(r_0,s_0)<p^kS^i(r_0,s_0).
  \]
  Therefore, if $S^i(r_0,s_0) \leq Q_3^i(r_0,s_0)$ then $Q_3^i(r,s) < 0$. 
\end{proof}

The next lemma gives us more conditions for when 
$T^i(r,s)\leq -Q_3^i(r,s)$ and 
$S^i(r,s) \leq -Q_3^i(r,s)$.
\begin{lem}\label{lem:reduce_to_base}
If $(r,s)$ satisfies $r +s = -1$ and $r >p-1$, then 
$S^i(r,s) \leq -Q_3^i(r,s)$ and $T^i(r,s) \leq -Q_3^i(r,s)$. 
\end{lem}
\begin{proof}
Assume first that $i=1$.  We consider the case where 
\[
(r_0,s_0) = (a_0 + pz, p-1-a_0 + p(-z-1)).
\] 
From (\ref{eq:fund_lineQ_3}) and (\ref{eq:base_caseQ_3}), we get that
\begin{align*}
-Q_3^1(r_0,s_0) &= \theta^1(z,-z-1,a_0) + \frac{p}{2}z(z+1)\\
                      &> \theta^1(z,-z-1,a_0) + \frac{1}{2}z(z+1)+1\\
                      &= T^1(r_0,s_0)
\end{align*}
and 
\begin{align*}
-Q_3^1(r_0,s_0) &= \theta^1(z,-z-1,a_0) + \frac{p}{2}z(z+1)\\
                      &> \theta^1(z,-z-1,a_0) + \frac{1}{2}z(z+1)\\
                      &= S^1(r_0,s_0).
\end{align*}
If we now consider weights of the form 
\[
(r_1,s_1) = (a_1 + pr_0, p-1-a_1 + ps_0)
\] 
where $(r_0,s_0)$ is as above. Then again by using (\ref{eq:fund_lineQ_3}) we can show that
\begin{align*}
-Q_3^1(r_1,s_1) &> T^1(r_1,s_1) \\
-Q_3^1(r_1,s_1) &> S^1(r_1,s_1).
\end{align*}
Continuing in this way by inductively writing 
\[
(r_n,s_n) = (a_{n-1} + pr_{n-1}, p-1-a_{n-1} + ps_{n-1}))
\]
we can prove that $S^i(r,s) \leq -Q_3^1(r,s)$ and $T^i(r,s) \leq -Q_3^1(r,s)$ whenever $r+s=-1$.

Finally, by observing that if $r+s=-1$, then $Q^1(r,s) = Q^2(r,s)$, $S^1(r,s) =T^2(r,s)$ and $T^2(r,s) = S^1(r,s)$, 
we immediately get that the inequalities hold for $i=2$ as well. 
\end{proof}

We get an additional vanishing result  for $Q_3^i(r,s)$. 
\begin{lem}\label{prop:non-base}
If  $(r,s) = (x,y)+p^k(r_0,s_0)$ is such that $a_{k-1} + b_{k-1} \neq p-1$,  $r_0+s_0 = -1$ and $r_0 \geq p$, then 
$Q_3^i(r,s) < 0$. 
\end{lem} 
\begin{proof}
  Follows immediately from Corollary~\ref{cor:reduce_to_base}.
\end{proof}

So we now know that if $Q_3^i(r,s) \geq 0$, then 
\[
(r,s) = (x,y) + p^k(z,-z-1)
\]
where $a_{k-1}+b_{k-1} \neq p-1$ and $1 \leq z \leq p-1$. 

Up until this point, we have always allowed the cohomology degree $i$ to be either $1$ or $2$. However, for the rest of this section
 we will fix $i=1$. This is justifiable, since by Serre duality if we can determine the support varieties for all
modules of the form $H^1(r,s)$ with $(r+1)(s+1) < 0$, we will have also determined the support varieties for all modules of the
form $H^2(r,s)$ with $(r+1)(s+1) <0$. 
We will also assume that our weights are of  the form 
$$(r,s) = (x,y) + p^k(z,-z-1)$$ where $1 \leq z \leq p-1$ and $a_{k-1}+b_{k-1} \neq p-1$.
With this in mind, we can state the next vanishing result. 
\begin{lem}\label{prop:neg-case}
If $(r,s)= (x,y) + p^k(r_0,s_0)$ where $a_{k-1}+b_{k-1} \neq  p-1$ and $r_0 + s_0=-1$,  then $Q^1_3(r,s)<0$ provided  $x +y < p^k -1$. 
\end{lem}
\begin{proof}
By  Lemma~\ref{lem:reduce_to_base}, we can assume that $(r_0,s_0) = (z,-z-1)$ where 
$1 \leq z \leq p-1$. Therefore,  by Lemma~\ref{prop:general_formula} 
\begin{align*}
Q_3^1(r,s) &= -(x+y-(p^k-1))S^1(z,-z-1) + p^kQ_3^1(z,-z-1) \\
                    &=  (-x-y +p^k-1 -p^k)\frac{1}{2}z(z+1)\\
                    &= (-x-y-1)\frac{1}{2}z(z+1) < 0. 
\end{align*}
\end{proof}

This proposition combines all of our vanishing results so far, to
precisely determine the weights $(r,s)$ such that $Q_3^1(r,s) =0$. 
\begin{prop}\label{prop:Q_3_vanish}
 We get $Q_3^1(r,s) = 0$ if and only if $(r,s) = (x,y) + p^k(z,-z-1)$ where $k\geq 1$, 
$x+y = 2p^k -p^{k-1}-1$,  $1 \leq z \leq p-1$ and 
$2p-2 = z(z+1)$.
If no such $z$ exists, then $Q^1_3(r,s) \neq 0$ for all weights $(r,s)$ with $(r+1)(s+1) <0$. 
\end{prop}
\begin{proof}
  By Lemmas~\ref{prop:non-base} and \ref{prop:neg-case}, we can assume that $x + y >p^k-1$ 
  and that $(r_0,s_0)=(z,-z-1)$ for some $1 \leq z \leq p-1$. Hence, by Lemma~\ref{prop:general_formula},
  \begin{align*}
    Q_3^1(r,s) &= (x+y -(p^k-1))T^1(z,-z-1) + p^kQ_3^1(z,-z-1). 
  \end{align*}
  If we assume $Q_3^1(r,s) = 0$ and set $h = \frac{1}{2}z(z+1)$ 
  so that $Q_3^1(z,-z-1) = -h$ and $T^1(z,-z-1) = h+1$, then 
  \begin{align*}
    x+y &= \frac{(2p^k-1)h}{h+1} + \frac{p^k-1}{h+1}\\
        &= \frac{(2p^k-1)(z^2 + z)}{z^2+z+2} + \frac{2p^k-2}{z^2+z+2} \\
        &= \frac{(z^z+z+1)2p^k - (z^2+z) -2}{z^2+z+2} \\
        &= \frac{2p^k(z^2+z+1)}{z^2+z+2} -1.
   \end{align*}
   We can establish that $z^2+z+2 \mid 2p^k(z^2+z+1)$, but since $z^2+z+1$ and $z^2+z+2$ are 
   coprime, we get $z^2+z+2 \mid 2p^k$ which implies
   $ \frac{1}{2}(z^2 + z + 2) \mid p.$
   In addition to $\frac{1}{2}(z^2+z+2) > 1$, since $p$ is prime we can prove that 
   \[ \frac{1}{2}(z^2+z+2) = p. \]
   Therefore, by rearranging the terms, this is equivalent to saying that 
   \[
     z(z+1) = 2p-2.
   \]
   In particular this yields
   \begin{align*}
     z^2+z+2 &= 2p\\
     z^2+z+1 &= 2p-1.
   \end{align*}
  If we plug it into the above expression for $x+y$, we get
  \[
    x+y = 2p^k-p^{k-1}-1.
  \]
\end{proof}

To summarize, we have shown that if  $H^1(a + pr, p-2-a + ps)|_{G_1}\neq 0$,
then it is not projective, except possibly when 
\[
  (r,s) = (x,y) + p^k(z,-z-1)
\]
where $2p-2 =z(z+1)$ and $(x,y) \in X_k(T)$ with $x+y = 2p^k-p^{k-1}-1$. 
In this case, it will be enough to show that $\delta^1(r,s-1) - \delta^1(r-1,s) \neq 0$ 
to get that $H^1(a+pr,p-2-a + ps)|_{G_1}$ is not projective. In this instance, we will need
to show that $R_3^1(r,s) \neq 0$.

\subsection{Calculations for $R_1^3(r,s)$}
The goal now is to show that for any weight of the form
\[
  (r,s) = (x,y) + p^k(z,-z-1)
\]
where $(x,y) \in X_k(T)$ with $x+y = 2p^k-p^{k-1}-1$ and $z(z+1)=2p-2$, 
\[
 R^1_3(r,s)=  \delta^1(r,s-1) - \delta^1(r-1,s) \neq 0.
\]

It would be helpful to characterize the weights of the form $(x,y) \in X_k(T)$ with $$x+y = 2p^k-p^{k-1}-1$$
for each $k \geq 1$. When $k=1$, we see that the only such weight is the Steinberg weight 
\[
  (x,y) = (p-1,p-1).
\]
For $k\geq 2$, we first observe that 
\[
  2p^k -p^{k-1}-1 = \left( \sum_{i=0}^{k-2}p^i(p-1) \right) + 2p^{k-1}(p-1).
\]
It can be verified that all such $(x,y)\in X_k(T)$
are of the form 
\[
  (x,y) = \left( \sum_{i=0}^{k-2}p^i(a_i,p-1-a_i) \right) + p^{k-1}(p-1,p-1)
\]
where $0\leq a_i \leq p-1$. We can see that there are $p^{k-1}$ possible choices for $(x,y)$. 

To prove this, we will use the recursive character formulas to  perform induction on $k$.
The following lemma will handle the case when $k=1$. 
\begin{lem}\label{lem:k=1_case}
  If $(r,s) =  (p-1,p-1) + p(z,-z-1)$ for $1 \leq z \leq p-1$, then 
  \[
    \delta^1(r,s-1)-\delta^1(r-1,s) > 0. 
  \] 
\end{lem}
\begin{proof}
  From \cite[Theorem 3.6]{and1979}, we know that 
  \[
    \delta^2(r,s-1) = \delta^2(r-1,s) = 0.
  \]
 Therefore, 
 \begin{align*}
   \delta^1(r,s-1) &= \frac{1}{2}((pz +p)(pz+1)(p-1) )\\
   \delta^1(r-1,s) &= \frac{1}{2}( (pz + p -1)(pz)(p-1) ),
 \end{align*}
and the difference is given by 
\[
  \delta^1(r,s-1)-\delta^1(r-1,s) = \frac{p(p-1)}{2}(2z+1)>0.
\]
\end{proof}
This key lemma will allow us to handle the case where $k \geq 2$. 
\begin{lem}\label{lem:inductive_step}
Let $(r,s) = (a,p-1-a)+p(r_0,s_0)$ where $0\leq a \leq p-1$ and $Q_3^1(r_0,s_0) =0$, so that
\[
  \delta^1(r_0,s_0) + \delta^1(r_0-1,s_0-1) = \delta^1(r_0,s_0-1) + \delta^1(r_0-1,s_0),
\]
then
\[
  \delta^1(r,s-1)-\delta^1(r-1,s) =p^2(\delta^1(r_0,s_0-1) -\delta^1(r_0-1,s_0)). 
\]
In particular, if $\delta^1(r_0,s_0-1) - \delta^1(r_0-1,s_0) >0$, then 
$\delta^1(r,s-1)-\delta^1(r-1,s) > 0$. 
\end{lem}
\begin{proof}
  First we consider the case where $1\leq a \leq p-2$, then 
  \begin{align*}
    \delta^1(r,s-1) &= \delta(a,p-2-a)(\delta^1(r_0,s_0)+\delta^1(r_0-1,s_0-1)) \\
                    &\,+\delta(p-1,a)\delta^1(r_0,s_0-1) + \delta(p-2-a,p-1)\delta^1(r_0-1,s_0)\\
    \delta^1(r-1,s) &= \delta(a-1,p-1-a)(\delta^1(r_0,s_0)+\delta^1(r_0-1,s_0-1)) \\
                    &\,+\delta(p-1,a-1)\delta^1(r_0,s_0-1) + \delta(p-1-a,p-1)\delta^1(r_0-1,s_0).
  \end{align*}
  Thus the difference is given by
  \begin{align*}
   &\delta^1(r,s-1)-\delta^1(r-1,s) = \\
   &\qquad (\delta(a,p-2-a)-\delta(a-1,p-1-a))(\delta^1(r_0,s_0) + \delta^1(r_0-1,s_0-1)) \\
   &\quad+(\delta(p-1,a)-\delta(p-1,a-1))\delta^1(r_0,s_0-1) \\
   &\quad+ (\delta(p-2-a,p-1)-\delta(p-1-a,p-1))\delta^1(r_0-1,s_0)
 \end{align*}
  where
  \begin{align*}
    \delta(a,p-2-a)-\delta(a-1,p-1-a)&=-\frac{p}{2}(2a+1-p)\\
    \delta(p-1,a) - \delta(p-1,a-1) &= \frac{p}{2}(2a+1+p) \\
    \delta(p-2-a,p-1)-\delta(p-1-a,p-1)&= \frac{p}{2}(2a+1-3p).
  \end{align*}
  By substituting 
  \[
    \delta^1(r_0,s_0)+\delta^1(r_0-1,s_0-1) = \delta^1(r_0,s_0-1)+\delta^1(r_0-1,s_0)
  \]
  and simplifying, we get 
  \[
    p^2(\delta^1(r_0,s_0-1)-\delta^1(r_0-1,s_0)).
  \]

  Now consider the case where $a=0$, 
  \begin{align*}
    \delta^1(r,s-1) &= \delta(0,p-2)(\delta^1(r_0,s_0)+\delta^1(r_0-1,s_0-1)) \\
                    &\,+\delta(p-1,0)\delta^1(r_0,s_0-1) + \delta(p-2,p-1)\delta^1(r_0-1,s_0)\\
    \delta^1(r-1,s) &= p^3\delta^1(r_0-1,s_0).
  \end{align*}
  So that the difference is 
  \begin{align*}
   &\delta^1(r,s-1)-\delta^1(r-1,s) = \\
   &\qquad \delta(0,p-2)(\delta^1(r_0,s_0) + \delta^1(r_0-1,s_0-1)) \\
   &\quad+\delta(p-1,0)\delta^1(r_0,s_0-1) + (\delta(p-2,p-1)-p^3)\delta^1(r_0-1,s_0).
 \end{align*}
 It is immediately verified upon simplification that 
  \[
    p^2(\delta^1(r_0,s_0-1)-\delta^1(r_0-1,s_0)).
  \]

  Finally, we consider the case when $a=p-1$ where 
  \begin{align*}
    \delta^1(r,s-1) &= p^3\delta^1(r_0,s_0-1) \\
    \delta^1(r-1,s) &= \delta(p-2,0)(\delta^1(r_0,s_0)+\delta^1(r_0-1,s_0-1)) \\
                    &\,+\delta(p-1,p-2)\delta^1(r_0,s_0-1) + \delta(0,p-1)\delta^1(r_0-1,s_0).
  \end{align*}
  Now we get 
  \begin{align*}
   &\delta^1(r,s-1)-\delta^1(r-1,s) = \\
   &\qquad -\delta(p-2,0)(\delta^1(r_0,s_0) + \delta^1(r_0-1,s_0-1)) \\
   &\quad+(p^3-\delta(p-1,p-2))\delta^1(r_0,s_0-1) - \delta(0,p-1)\delta^1(r_0-1,s_0).
 \end{align*}
 Again, it can be verified that when this expression is simplified, we obtain 
 \[
    p^2(\delta^1(r_0,s_0-1)-\delta^1(r_0-1,s_0)).
  \]
\end{proof}

We now have enough to prove the necessary non-vanishing result for $R_3^1(r,s)$. 
\begin{prop}\label{prop:R_3_vanish}
Let $(r,s) = (x,y) + p^k(z,-z-1)$ with $(x,y) \in X_k(T)$, $x+y = 2p^k-p^{k-1}-1$ and $z(z+1)=2p-2$,
then
\[
  \delta^1(r,s-1) - \delta^1(r-1,s) \neq 0.
\]
\end{prop}
\begin{proof}
  The case when $k=1$ was handled earlier in Lemma~\ref{lem:k=1_case}, so we may assume that $k\geq 2$. 
  We observe that 
  \[
    (r,s) = (a,p-1-a)+p(r_0,s_0)
  \]
  where $(r_0,s_0)=(x_0,y_0) + p^{k-1}(z,-z-1)$ with $x_0+y_0=2p^{k-1}-p^k -1$.
  Thus if we assume that $\delta^1(r_0,s_0-1) -\delta^1(r_0-1,s_0) > 0$, then 
  by Lemma~\ref{lem:inductive_step} we must have that 
  \[
   \delta^1(r,s-1)-\delta^1(r-1,s) > 0.
 \]
 We are finished by induction. 
\end{proof}

\end{document}